\documentclass[10pt,twoside,final]{siamltex}

\usepackage{amssymb}
\usepackage{amsmath,xcolor}
\usepackage{mathrsfs}
\usepackage{graphicx}
\usepackage{calc}%

\usepackage{cite}
\usepackage[none]{hyphenat}
\usepackage[notref,notcite]{showkeys}
\usepackage{extarrows}

\newcommand{\Mm}{\mathcal M}
\newcommand{\Dd}{\mathcal D}
\newcommand{\sign}{{\rm sign}}
\newcommand{\varep}{\varepsilon}
\newcommand{\gabs}[1]{\left | #1 \right|}
\newcommand{\norm} [1]{\left\| {#1}\right\|}

\newcommand{\intb}[1]{\left\langle #1 \right\rangle}
\newcommand{\intd}[1]{\left( #1 \right)}

\newcommand{\ddt}{\frac d{dt} }
\newcommand{\eqdef}{\xlongequal{\text{def}}}%

\def\beq{\begin{equation}}
\def\eeq{\end{equation}} 

\def\beqs{\begin{equation*}}
\def\eeqs{\end{equation*}}

\def\bals{\begin{align*}}
\def\eals{\end{align*}}

\def\bspl{\begin{split}}
\def\espl{\end{split}}

\def\myclearpage{}
\title{GALERKIN FINITE ELEMENT METHOD FOR GENERALIZED
FORCHHEIMER EQUATION OF SLIGHTLY COMPRESSIBLE FLUIDS IN POROUS MEDIA}
\author{Thinh Kieu \footnotemark[2]}

\begin{document}

\maketitle
 
\renewcommand{\thefootnote}{\fnsymbol{footnote}}
\footnotetext[2]{Department of Mathematics, University of North Georgia, Gainesville Campus, 3820 Mundy Mill Rd., Oakwood, GA 30566, U.S.A. ({\tt thinh.kieu@ung.edu}).}               
               
\begin{abstract}  We consider the generalized Forchheimer flows for slightly compressible fluids. Using Muskat's and Ward's general form of Forchheimer equations, we describe the fluid dynamics by a nonlinear degenerate parabolic equation for the density. 
 We study Galerkin finite elements method for the initial boundary value problem. The existence and uniqueness of the approximation are proved. The prior estimates for the solutions in $L^\infty(0,T;L^q(\Omega)), q\ge 2$, time derivative in $L^\infty(0,T;L^2(\Omega))$ and gradient in $L^\infty(0,T;W^{1,2-a}(\Omega)),$ with $a\in (0,1)$ are established. Error estimates for the density variable are derived in several norms for both continuous and discrete time procedures. Numerical experiments using backward Euler scheme confirm the theoretical analysis regarding convergence rates.  
 \end{abstract}
            
            
            
\begin{keywords}
Porous media, immersible flow, error analysis, Galerkin finite element, nonlinear degenerate parabolic equations, generalized Forchheimer equations, numerical analysis.
\end{keywords}

\begin{AMS}
65M12, 65M15, 65M60, 35Q35, 76S05.
\end{AMS}

\pagestyle{myheadings}
\thispagestyle{plain}
\markboth{Thinh Kieu}{Galerkin finite element method for Forchheimer equation of slightly compressible fluids}
            
            
\myclearpage    
\section {Introduction }

Fluid flow through porous media has been studied in many fields such as chemistry, physics, engineering, and geology. The most common equation to describe fluid flows in porous media is the Darcy law   
\beq\label{Darcy}
-\nabla p = \frac {\mu}{\kappa} v,
\eeq
where $p$, $v$, $\mu$, $\kappa$ are, respectively (resp.), the pressure, velocity, absolute viscosity and permeability. 

For the Reynolds number large, for instance the high-velocity flow \cite{Muskatbook,BearBook}, Darcy's law is not valid any more. A nonlinear relationship between velocity and gradient of pressure is introduced, as suggested by Forchheimer in \cite{Forchh1901,ForchheimerBook},  by adding the higher order term of velocity to the Darcy law. It is known as generalized Forchheimer laws, see \cite{Muskatbook,Ward64,BearBook,NieldBook,StraughanBook} and references therein. Forchheimer established the following three nonlinear empirical models:  
\beq\label{2term}
-\nabla p=av+b|v|v,
\quad
-\nabla p=av+b|v|v+c |v|^2 v,
\quad
-\nabla p=av+d|v|^{m-1}v, m\in(1, 2).
\eeq
Above, the positive constants $a,b,c,d$ are obtained from experiments.

The generalized Forchheimer equation of \eqref{Darcy} and \eqref{2term} were proposed in \cite{ABHI1,HI1,HI2} of the form 
\beq\label{gF}
-\nabla p =\sum_{i=0}^N a_i |v|^{\alpha_i}v. 
\eeq  

These equations are analyzed numerically in \cite{Doug1993,EJP05,K1},
theoretically in \cite{ABHI1,HI2,HIKS1,HKP1,HK1,HK2} for single phase flows, and also in \cite{HIK1,HIK2} for two phase flows.

In order to take into account the the dependence on density in generalized Forchheimer equation, we modify \eqref{gF} using dimension analysis by Muskat \cite{Muskatbook} and Ward \cite{Ward64}. They proposed the following equation for both laminar and turbulent flows in porous media:
\beq\label{W}
-\nabla p =F(v^\alpha \kappa^{\frac {\alpha-3} 2} \rho^{\alpha-1} \mu^{2-\alpha}),\text{ where  $F$ is a function of one variable.}
\eeq 
In particular, when $\alpha=1,2$, Ward \cite{Ward64} established from experimental data that
\beq\label{FW} 
-\nabla p=\frac{\mu}{\kappa} v+c_F\frac{\rho}{\sqrt \kappa}|v|v,\quad \text{where }c_F>0.
\eeq

Combining  \eqref{gF} with the suggestive form \eqref{W} for the dependence on $\rho$ and $v$, we propose the following equation 
 \beq\label{FM}
-\nabla p= \sum_{i=0}^N a_i \rho^{\alpha_i} |v|^{\alpha_i} v,
 \eeq
where $N\ge 1$, $\alpha_0=0<\alpha_1<\ldots<\alpha_N$ are real numbers, the coefficients $a_0, \ldots, a_N$ are positive.  
Here, the viscosity and permeability are considered constant and we do not specify the dependence of $a_i$'s on them.

Multiplying both sides of previous equation to $\rho$, we obtain 
 \beq\label{eq1}
 g( |\rho v|) \rho v   =-\rho\nabla p,
 \eeq
where the function $g$ is a generalized polynomial with non-negative coefficients. More precisely,  the function $g:\mathbb{R}^+\rightarrow\mathbb{R}^+$ is of the form
\beq\label{eq2}
g(s)=a_0s^{\alpha_0} + a_1s^{\alpha_1}+\cdots +a_Ns^{\alpha_N},\quad s\ge 0, 
\eeq 
where $N\ge 1,\alpha_0=0<\alpha_1<\ldots<\alpha_N$ are fixed real numbers, the coefficients $a_0, \ldots, a_N$ are non-negative numbers with $a_0>0$ and $a_N>0$.  

The equation of state which, for slightly compressible fluids, is
\beqs
\frac 1\rho \frac{d\rho}{dp}=\frac 1\kappa=const.>0.
\eeqs
Hence
\beq \label{eq3}
\nabla \rho = \frac{1}{\kappa} \rho\nabla p, 
\quad \text{ or }\quad  \rho \nabla p=\kappa\nabla \rho.
\eeq
Combining \eqref{eq2} and \eqref{eq3} implies that  
 \beq
 g( |\rho v|) \rho v   =-\kappa\nabla \rho .
 \eeq
By rescaling coefficients of $g$ we have
\[
 g( |\rho v|) \rho v   =-\nabla \rho.
\]
Hence
\beq\label{ru} 
\rho v=- K(|\nabla \rho |)\nabla \rho,
\eeq
where the function $K: \mathbb{R}^+\rightarrow\mathbb{R}^+$ is defined for $\xi\ge 0$ by
\beq\label{Kdef}
K(\xi)=\frac{1}{g(s(\xi))}, \text{ with }  s=s(\xi) \text{  being the unique non-negative solution of } sg(s)=\xi.
\eeq

The continuity equation is
\beq\label{con-law}
\phi\rho_t+{\rm div }(\rho v)=f.
\eeq
where constant  $\phi\in(0,1)$ is the porosity, $f$ is external mass flow rate . 

By combining \eqref{ru} and \eqref{con-law} we obtain
\beq\label{utporo}
\phi\rho_t-\nabla \cdot(K(|\nabla \rho|)\nabla \rho) = f.
\eeq
Then by scaling the time variable in \eqref{utporo} implies   
\beq\label{maineq}
\rho_t- \nabla\cdot(K(|\nabla \rho|)\nabla \rho)=f.
\eeq

The system of equations describing the fluid motion reduces to a scalar equation of
density function. This is a nonlinear parabolic equation which degenerates as the gradient of density  approach to infinity. 

For the existence and regularity theory of degenerate parabolic equations, see e.g. \cite{MR2566733, LadyParaBook68,HIKS1}. The popular numerical method for modeling flow in porous media are mixed finite element method. Arbogast, Wheeler and Zhang \cite {ATWZ96} first analyzed mixed finite element approximations of degenerate parabolic equation arising in flow in porous media. Not so long later,  Woodward and Dawson in \cite{WCD00} study of expanded mixed finite element methods for a nonlinear parabolic equation modeling flow into variably saturated porous media. In recently years, the Galerkin finite element method for a coupled nonlinear degenerate system of advection-diffusion equations has studied in \cite{F14,F10,F06,F07,FS04,FS95}. In their analysis, the Kirchhoff transformation is used to move the nonlinearity from coefficient $K$ to the gradient and thus simplifies analysis of the equations. This transformation does not applicable for the equation \eqref{maineq}.  

In this paper,  we combine techniques in \cite {HI1,HI2,HIKS1,HK1,HK2,HKP1} utilizing the special~structures of equation to obtain the error estimates for the solutions in several norms of interest. We obtain these results without any extra regularity assumption, though the order of error estimates are far from optimal order.      

 The paper is organized as follows:  
  In \S \ref{Intrsec} we introduce notations, relevant results in \cite{ABHI1,HI1,K1} and suitable trace estimates in Lemma \ref{traceest} for the nature of our equation.   
 In \S \ref{GalerkinMethod} we consider the semidiscrete finite element  Galerkin approximation and an implicit backward difference time discretization to solve problem \eqref{rho:eq}. 
  In \S \ref{Bsec} we establish many bounds for solutions, its time derivative and gradient  to problem \eqref{weakform} and \eqref{semidiscreteform} in Lebesgue norms. 
In \S \ref{errSec}  we analyze two version of a Galerkin finite element approximations, the continuous Galerkin method and  the discrete Galerkin method. Using the monotonicity properties of Forchheimer equation and the boundedness of solutions, the {\it priori} error estimates are derived for solution in $L^q$-norm with $2\le q\le \infty$  and for gradient of solution in $L^{2-a}$-norm. 
Finally, in \S \ref{Num-result} numerical examples using the Lagrange elements of order $2$ are carried out in two-dimension. The results strongly support our theoretical analysis regarding convergence rates.  


\section{Notations and preliminary results}\label{Intrsec}

Suppose that $\Omega$ is an open, bounded subset of $\mathbb{R}^d$, with $d=2,3,\ldots$, and has $C^1$-boundary $\Gamma=\partial \Omega$. Let $L^2(\Omega)$ be the set of square integrable functions on $\Omega$ and $( L^2(\Omega))^d$ the space of $d$-dimensional vectors which have all components in $L^2(\Omega)$.  We denote $(\cdot, \cdot)$ the inner product in either $L^2(\Omega)$ or $(L^2(\Omega))^d$ that is
\beqs
( \xi,\eta )=\int_\Omega \xi\eta dx \quad \text{ or } (\boldsymbol{\xi},\boldsymbol \eta )=\int_\Omega \boldsymbol{\xi}\cdot \boldsymbol{\eta} dx. 
\eeqs      

The notation $\intb{\cdot ,\cdot}$ will be used for the $L^2(\partial \Omega)$ inner-product and $ \norm{u}_{L^p}=\norm{u}_{L^p(\Omega)}$ for standard  Lebesgue norm of the measurable function.    
 The notation $\norm {\cdot}$ will means scalar norm $\norm{\cdot}_{L^2(\Omega)}$ or vector norm $\norm{\cdot}_{(L^2(\Omega))^d}$. We denote $\norm{u}_{L^p(L^q)} =\norm{u}_{L^p(0,T; L^q(\Omega))}, 1\le p,q<\infty$ means the mixed Lebesgue norm for a function $u$ while $\norm{u}_{L^p(H^q)}=\norm{u}_{L^p(0,T;H^q(\Omega))}, 1\le p,q<\infty$ stand for the mixed Sobolev-Lebesgue norm of a function $u$. 

For $1\le q\le +\infty$ and $m$ any nonnegative integer, let
\beqs
W^{m,q}(\Omega) = \big\{u\in L^q(\Omega), D^q u\in L^q(\Omega), |q|\le m \big\}
\eeqs 
denote a Sobolev space endowed with the norm 
\beqs
\norm{u}_{m,q} =
\Big( \sum_{|\alpha|\le m} \norm{D^\alpha u}^q_{L^q(\Omega)} \Big)^{\frac 1 q}.
\eeqs    
Define $H^m(\Omega)= W^{m,2}(\Omega)$ with the norm $\norm{\cdot}_m =\norm{\cdot }_{m,2}$. 

Throughout this paper, we use short hand notations, 
\beqs
I=(0,T), \quad \norm{\rho(t)} = \norm{ \rho(\cdot, t)}_{L^2(\Omega)}, \forall t\in I \quad \text{
 and } \quad \rho^0(\cdot) =  \rho(\cdot,0).
 \eeqs
 
 Also our calculations frequently use the following exponents
\beq\label{a-const }
   a=\frac{\alpha_N}{\alpha_N+1}= \frac{\deg (g)}{\deg (g)+1},
  \eeq
  and
  \beq\label{MLGdef}
  \beta =2-a,\quad \lambda = \frac {2-a}{1-a}=\frac{\beta}{\beta-1},\quad \gamma=\frac{a}{2-a}=\frac a{\beta}.
  \eeq  

The arguments $C, C_0, C_1,\ldots$ will represent for positive generic constants and their values  depend on exponents, coefficients of polynomial  $g$,  the spatial dimension $d$ and domain $\Omega$, independent of the initial and boundary data, size of mesh and time step. These constants may be different place by place.

\begin{lemma}[cf. \cite{ABHI1, HI1}]
 The function $K(\xi)$ has the following properties
 
(i) $K: [0,\infty)\to (0,a_0^{-1}]$ and it decreases in $\xi,$  

 (ii) For any $n\ge 1$, the function $K(\xi)\xi^n$ increasing and $K(\xi)\xi^n\ge 0$

(iii) Type of degeneracy  \beq\label{i:ineq1}  \frac{c_1}{(1+\xi)^a}\leq K(\xi)\leq \frac{c_2}{(1+\xi)^a},  \eeq

(iv) For all $n\ge 1,$ \beq\label{i:ineq2} c_3(\xi^{n-a}-1)\leq K(\xi)\xi^n\leq c_2\xi^{n-a}, \eeq

 (iv) Relation with its derivative \beq\label{i:ineq3} -aK(\xi)\leq K'(\xi)\xi\leq 0, \eeq
  where $c_1, c_2, c_3$ are positive constants depending on $\Omega$ and $g$. 
  
\end{lemma}
We define 
\beq\label{Hdef}
H(\xi)=\int_0^{\xi^2} K(\sqrt{s}) dx, \text{~for~} \xi\geq 0. 
\eeq
The function $H(\xi)$ can compare with $\xi$ and $K(\xi)$ by
\beq\label{i:ineq4}
K(\xi)\xi^2 \leq H(\xi)\leq 2K(\xi)\xi^2. 
\eeq


For the monotonicity and continuity of the differential operator in \eqref{maineq} we have the following results. 
\begin{lemma}[cf. \cite{HI1}] One has 

{\rm (i)} For all $y, y' \in \mathbb{R}^d$, 
\beq\label{Qineq}
\big(K(|y'|)y' -K(|y|)y \big)\cdot(y'-y)\geq (\beta-1)K( \max\{ |y|, |y'|\} )|y' -y|^2 .
\eeq      

 {\rm (ii)} For the vector functions $s_1, s_2$, there is a positive constant $C$  such that
\beq\label{Mono}
 \big( K(|s_1|)s_1-K(|s_2|)s_2,s_1 -s_2\big)\geq C\omega\norm{s_1-s_2}_{L^{\beta}(\Omega)}^2,
\eeq
where 
$\omega =\big(1+ \max\{\|s_1\|_{L^{\beta}(\Omega)} ,  \|s_2\|_{L^{\beta}(\Omega)} \}\big)^{-a}. $
\end{lemma}

\begin{lemma}[cf. \cite{K1}] \label{Lips}  For all $y, y' \in \mathbb{R}^d$. There exist a positive constant $C$ depending on polynomial  $g$, the spatial dimension $d$ and domain $\Omega$ such that 
\beq\label{Kcont}
   \left|K(|y'|)y' -K(|y|)y\right| \leq C|y' -y|.
\eeq  
\end{lemma}
Next we derive trace estimates suitable to our nonlinear problem.
\begin{lemma}\label{traceest}
(i) Let $q\ge 2$. Assume $v(x)$ is a function defined on $\Omega$. If $|v|^{q-1}\in W^{1,1}(\Omega)$ then for all $\varep>0$,      
 \beq\label{sec2}
\int_\Gamma |v|^{q-1}d\sigma\le C\norm{v}_{L^q(\Omega)}^{q-1} + \varep\int_\Omega |v|^{q-2}|\nabla v|^{\beta} dx +C\varep^{-\frac 1{\beta-1}}\norm{v}_{L^q(\Omega)}^{q-2}.
\eeq

(ii) If $u\in L^\infty(\Gamma)$ and $|v|\in W^{1,1}(\Omega)$ then for all $\varep>0,$
\beq\label{bnd-eps}
\gabs{ \langle u, v\rangle } \le \varep \left(\norm{v}^2+ \norm{\nabla v}_{L^{\beta}(\Omega)}^{\beta}\right)+C\Big( \varep^{-1}\norm{u}_{L^\infty(\Gamma)}^2 +\varep^{-\frac 1{\beta-1}} \norm{u}_{L^\infty(\Gamma)}^{^\lambda}\Big).
\eeq
In particular case, 
\beq\label{bnd-est}
\gabs{ \langle u, v\rangle } \le \frac 14 \left(\norm{v}^2+ \norm{\nabla v}_{L^{\beta}(\Omega)}^{\beta}\right)+C\Big( 1 + \norm{u}_{L^\infty(\Gamma)}^{^\lambda}\Big).
\eeq
Above, $C$ is a positive constant independent of  $u,$ $v,$ $\varep.$
\end{lemma}
\begin{proof} We recall the trace theorem 
\beqs
\int_\Gamma |\phi| dx \le C\int_\Omega |\phi| dx +C\int_\Omega |\nabla \phi|dx, 
\eeqs
for all $\phi\in W^{1,1}(\Omega),$ where $C$ are positive constants depending on $\Omega$. Applying the Trace theorem to $\phi=|v|^{q-1}$ shows that   
\beq\label{trace1}
\int_\Gamma |v|^{q-1}d\sigma\le C\int_\Omega |v|^{q-1}dx + C\int_\Omega |v|^{q-2}|\nabla v| dx.
\eeq
Note that $\frac 1\beta + \frac1\lambda=1$ and $ \frac \lambda \beta =\frac 1{\beta-1} $. Using Young's inequality with exponent $\beta$ and $\lambda$, we find that for $\varep>0$
\beq\label{Y1}
  \begin{aligned}
 |v|^{q-2}|\nabla v| &= \varep^{\frac1{\beta}}|\rho_h|^{\frac{q-2}{\beta}}|\nabla v| \cdot \varep^{-\frac1{\beta}}|v|^{\frac{q-2}{\lambda}}\\
 &\le  \varep |v|^{q-2}|\nabla v|^{\beta}+C\varep^{-\frac 1{\beta-1}}  |v|^{q-2}. 
 \end{aligned} 
 \eeq
 Combining \eqref{trace1}, \eqref{Y1} we obtain  
 \beq\label{aa}
 \int_\Gamma |v|^{q-1}d\sigma\le C\int_\Omega |v|^{q-1}dx + \varep \int_\Omega |v|^{q-2}|\nabla v|^{\beta}dx+C\varep^{-\frac 1{\beta-1}}  \int_{\Omega} |v|^{q-2}  dx.
 \eeq
 
Inequality  \eqref{sec2} follows by using H\"older's inequality to the first term of the right hand side in \eqref{aa}.

(ii) We have 
\beq
\gabs{ \langle u, v\rangle }
\le \norm{u}_{L^\infty(\Gamma)}\int_\Gamma |v|d\sigma.
\eeq
For all $\delta>0,$ using \eqref{aa} with $q=2$ gives, 
\beqs
\int_\Gamma |v|d\sigma\le C\int_{\Omega} |v| dx  + \delta\norm{\nabla v}_{L^{\beta}}^{\beta} +C\delta^{-\frac 1{\beta-1 }}.
\eeqs 
Applying Young's inequality leads to 
\beqs
\int_\Gamma |v|d\sigma\le  \delta\norm{v}^2+C\delta^{-1} + \delta\norm{\nabla v}_{L^{\beta}}^{\beta} +C\delta^{-\frac 1{\beta-1 }},
\eeqs
which gives 
\beqs
\gabs{ \langle u, v\rangle } \le \norm{u}_{L^\infty(\Gamma)}\Big( \delta\norm{v}^2+C\delta^{-1} + \delta\norm{\nabla v}_{L^{\beta}}^{\beta} +C\delta^{-\frac 1{\beta-1}}\Big).
\eeqs

If $\norm{u}_{L^\infty(\Gamma)} =0$  then \eqref{bnd-eps} clearly holds true.

Otherwise, selecting $\delta = \varep  \norm{u}_{L^\infty(\Gamma)}^{-1} $ and the fact that $\frac 1{\beta-1} + 1=\lambda$  we obtain estimate \eqref{bnd-eps}.

Estimate \eqref{bnd-est} follows by choosing $\varep = 1/4$ in \eqref{bnd-eps} and using Young's inequality.  
\end{proof}

\section{The Galerkin finite element method}\label{GalerkinMethod}
Our aim is to study equation \eqref{maineq} for density of slightly compressible fluids in bounded
domain in porous media. The fluid flows are subject to some conditions on the boundary. 


We consider the initial boundary value problem associated with \eqref{maineq} ,
\beq\label{rho:eq}
\begin{cases}
\rho_ t - \nabla \cdot (K (|\nabla \rho|)\nabla \rho  ) =f &\text {in }  \Omega\times I,\\
\rho(x,0)=\rho^0(x) &\text {in } \Omega,\\
K(|\nabla \rho|)\nabla \rho \cdot \nu +\psi=0  &\text{on } \Gamma \times I,
\end{cases}
\eeq 
where $\rho^0(x)$ and $\psi(x,t)$ are given initial and boundary data, respectively.

The variational formulation of \eqref{rho:eq} is defined as the following: 

Find $\rho :[0,T] \rightarrow  W \equiv H^1(\Omega)  $ such that 
\beq\label{weakform}
(\rho_t, w) + (K(|\nabla \rho |)\nabla \rho,\nabla w) =-\langle \psi, w\rangle + (f, w),  \quad w\in W
\eeq 
with $\rho(x,0)=\rho^0(x).$ 

\vspace{0.2cm}
Let  $\{\mathcal T_h\}_h$ be a family of quasi-uniform triangulations of $\Omega$ with $h$ being the maximum diameter of the element. Let $W_h$ the space of discontinuous piecewise polynomials of degree $r\ge 0$ over  $\mathcal T_h$.     
It is frequently valuable to decompose the analysis of the convergence of finite element methods by passing through a projection of the solution of the differential problem into the finite element space. Here we use the standard $L^2$-projection operator  (see \cite{Ciarlet78})  
$\pi: W \rightarrow W_h$,   satisfying
\begin{align*}
( \pi w ,   v_h ) = ( w ,   v_h ), \quad &\forall w\in W, v_h \in W_h.
\end{align*}
These projections have well-known approximation properties (see~\cite{BF91,JT81}). 

(i) $\norm{\pi w}\le \norm{w}$ holds for all $w\in L^2(\Omega)$.
   
(ii) There exist positive constants $C_1, C_2$ such that
\begin{equation}
\label{prjpi}
\begin{split}
\norm{\pi w - w }_{L^q(\Omega)} \leq C_1 h^m \norm{w}_{W^{m,q}(\Omega)} 
\end{split}
\end{equation}
for all $w \in W^{m,q}(\Omega)$,  $0\le m \le r+1, 1\le q \le \infty$. Notation $\norm {\cdot}_{m,q}$ denotes a standard norm in Sobolev space $W^{m,q}(\Omega)$. In short hand, when $q=2$ we write \eqref{prjpi} as   
\beqs
\norm{\pi w - w } \leq C_1 h^m \norm{w}_{m} \quad  \text { and }\quad \norm{ \pi w - w }_{L^\infty(\Omega)} \leq C_2 h^m \norm{w}_{m+1}. 
\eeqs

 Replacing the original density by its approximation, the semidiscrete formulation of~\eqref{weakform} can read as following: Find $\rho_h:[0,T] \rightarrow W_h$ such that
\beq\label{semidiscreteform}
  (\rho_{h,t},w_h)+(K(|\nabla \rho_h|)\nabla \rho_h,\nabla w_h) =-\langle \psi, w_h\rangle + (f, w_h),  \quad w_h\in W_h
\eeq 
with initial data $\rho_h^0=\pi \rho^0(x)$.

Let $N$ be the positive integer, $t_0=0 < t_1 <\ldots < t_N= T$ be partition interval $[0,T]$ of $N$ sub-intervals, and let $\Delta t = t_{n} - t_{n-1}=T/N$  be the $n$-th time step size, $t_n=n\Delta t$ and $\varphi^n = \varphi(\cdot, t_n)$. 
The discrete time Galerkin finite element approximation to \eqref{weakform} is defined as follows:  

Find $ \rho_h^n\in W_h$, $n=1,2,\dots, N$, such that 
\beq\label{fullydiscreteform}
\Big( \frac{ \rho_h^n - \rho_h^{n-1}}{\Delta t }, w_h\Big) +  \big(K(|\nabla \rho_h^n|)\nabla \rho_h^n, \nabla w_h\big) =-\langle \psi^n, w_h\rangle  +(f^n, w_h ), \quad \forall w_h\in W_h, 
\eeq
with initial data are chosen as follows:
$
\rho_h^0(x)=\pi \rho^0(x). 
$

\section{A priori estimate for solutions}\label{Bsec}
We study the \eqref{weakform}, \eqref{semidiscreteform} and \eqref{fullydiscreteform} equations for the density with fixed the functions $g(s)$ in \eqref{eq1} and \eqref{eq2}. 
Therefore, the exponents $\alpha_i$ and coefficients $a_i$ are all fixed, and so are the function $K(\xi)$, $H(\xi)$  in \eqref{Kdef}, \eqref{Hdef}.  

With properties \eqref{i:ineq1}, \eqref{i:ineq2}, \eqref{i:ineq3}, the monotonicity \eqref{Qineq}, and by classical theory of monotone operators \cite{MR0259693,s97,z90}, the authors in \cite {HIKS1} proved the global existence and uniqueness of weak solution of equation \eqref{weakform}. Moreover $\rho\in C(0,T,L^q(\Omega))\cap L_{loc}^{\beta}(0,T,W^{1,\beta}(\Omega))$, $q\ge 1$ and  $\rho_t \in L^{\beta'}_{loc}\bigl(0,T, (W^{1,\beta}(\Omega))'\bigr)\cap  L^2_{loc}\bigl(0,T, L^2(\Omega)\bigr)$  provided the initial, boundary data and $f$  sufficiently smooth. For  a { \it priori } estimate, we assume throughout this  paper  that $\psi, \psi_t$ belong to $C([0,T], L^\infty(\Gamma)),$ $\rho^0\in L^\infty(\Omega)$ and $f\in C([0,T], L^\infty(\Omega)) $ and the weak solution are sufficiently regularities both in $x$ and $t$ variables such that our calculations can be performed legitimately. 

The local existence of of approximate solution of \eqref{semidiscreteform} follows from Peano's theorem. The global existence of $\rho_h$ relies on the known theory of ordinary differential equation and stability estimate which shows as follows.

\begin{theorem}\label{bound-lq}
Let $q \ge 2$ and $\rho_h$ be a solution of \eqref{semidiscreteform}. Then there exist positive constant $C$ such that
\beq\label{res1a}
\norm{\rho_h}_{L^\infty(I,L^q(\Omega))} \le C\norm{\rho^0}_{L^q(\Omega)} +C\left[\int_0^T \Big(1+\norm{\psi}_{L^\infty(\Gamma)}^{\frac{q\lambda}{2} } + \norm{f}_{L^q(\Omega)}^q\Big) dt\right]^{\frac 1 q}.
\eeq
\end{theorem}
\begin{proof}
 Selecting $w_h= |\rho_h|^{q-1} \sign (\rho_h) $ in \eqref{semidiscreteform}, we obtain
\beq\label{dpalpha}
 \begin{split}
\frac{1}{q}\frac{d }{dt}\norm{\rho_h}_{L^q}^q &=-(q-1)\intd{ K(|\nabla \rho_h|)|\nabla \rho_h|^2, |\rho_h|^{q-2}}
\\
& -\intb{\psi, |\rho_h|^{q-1} \sign (\rho_h)} + \intd{f, |\rho_h|^{q-1} \sign (\rho_h)}.
\end{split}
\eeq
It follows from \eqref{i:ineq2} and H\"older's inequality that
\beq\label{term1}
\begin{aligned}
-(q-1)\intd{ K(|\nabla \rho_h|)|\nabla \rho_h|^2, |\rho_h|^{q-2}}&\le -c_0\int_\Omega ( |\nabla \rho_h|^{\beta}-1)|\rho_h|^{q-2}dx\\
&= -c_0 \int_\Omega |\nabla \rho_h|^{\beta} |\rho_h|^{q-2}  dx+c_0\int_\Omega |\rho_h|^{q-2}dx\\ 
& \le -c_0 \int_\Omega |\nabla \rho_h|^{\beta}|\rho_h|^{q-2}  dx+C\norm{\rho_h}_{L^q}^{q-2}, 
\end{aligned}
\eeq
where $c_0 =c_3(q-1).$\\
According to \eqref{sec2} in Lemma \ref{traceest}, 
\begin{multline}\label{term2}
\gabs{-\intb{\psi, |\rho_h|^{q-1} \sign (\rho_h)} }
\le \norm{\psi}_{L^\infty(\Gamma)}\int_\Gamma |\rho_h|^{q-1}d\sigma\\
\le C\norm{\psi}_{L^\infty(\Gamma)}\Big\{  \norm{\rho_h}_{L^q}^{q-1} + \varep\int_\Omega |\rho_h|^{q-2}|\nabla \rho_h|^{\beta} dx +\varep^{-\frac 1{\beta-1}}\norm{\rho_h}_{L^q}^{q-2}   \Big\}.
\end{multline}
Using Young's inequalities with exponent $q/(q-1)$ and  $q$, we obtain
\beq\label{term3}
\begin{aligned}
 \intd{f, |\rho_h|^{q-1} \sign (\rho_h)} \le C\norm{\rho_h}_{L^q}^{q-1} +C\norm{f}_{L^q}^q.
\end{aligned}
\eeq

Combining \eqref{term1}, \eqref{term2}, \eqref{term3} and \eqref{dpalpha} gives
\begin{multline}\label{diff-neq}
 \frac 1 q \frac{d }{dt}\norm{\rho_h}_{L^q}^q\le - c_0 \int_\Omega |\nabla \rho_h|^{\beta} |\rho_h|^{q-2}  dx +C\norm{\rho_h}_{L^q}^{q-2}\\
+C\norm{\psi}_{L^\infty(\Gamma)}\Big\{\norm{\rho_h}_{L^q}^{q-1} + \varep\int_\Omega |\rho_h|^{q-2}|\nabla \rho_h|^{\beta} dx +\varep^{-\frac 1{\beta-1}}\norm{\rho_h}_{L^q}^{q-2}\Big\}\\
+C\norm{\rho_h}_{L^q}^{q-1} +C\norm{f}_{L^q}^q.
\end{multline}
The case $\norm{\psi(t)}_{L^\infty(\Gamma)} =0$, by the means of Young's inequality in \eqref{diff-neq} we obtain
\beq\label{}
 \frac{d }{dt}\norm{\rho_h}_{L^q}^q+  qc_0 \int_\Omega |\nabla \rho_h|^{\beta} |\rho_h|^{q-2}  dx \le C\norm{\rho_h}_{L^q}^q +C(1+\norm{f}_{L^q}^q).
\eeq
 Applying Gronwall's inequality to above differential inequality implies that.
  \beq\label{re1}
  \begin{split}
\norm{\rho_h}_{L^\infty(I,L^q(\Omega))}^q &+qc_0\int_0^t\int_\Omega e^{C(t-\tau)} |\nabla \rho_h|^{\beta} |\rho_h|^{q-2}  dxd\tau \\
 &\qquad \qquad\le C\norm{\rho_{h}^0}_{L^q(\Omega)}^q +C\int_0^T \Big(1+ \norm{f}_{L^q(\Omega)}^q\Big) dt.
\end{split}
\eeq
 
Consider $\norm{\psi(t)}_{L^\infty(\Gamma)} \neq 0$. Selecting $\varep=\frac{c_0}{2C}  \norm{\psi}_{L^\infty(\Gamma)}^{-1}$ in \eqref{diff-neq}, we obtain
\beq\label{r:est}
\begin{split}
\frac{d }{dt}\norm{\rho_h}_{L^q}^q &\le - \frac{qc_0} 2 \int_\Omega |\nabla \rho_h|^{\beta} |\rho_h|^{q-2}  dx +C\norm{\rho_h}_{L^q}^{q-2}\\
&+C\norm{\psi}_{L^\infty(\Gamma)}\Big\{\norm{\rho_h}_{L^q}^{q-1}   + \norm{\psi}_{L^\infty(\Gamma)}^{\frac 1{\beta-1}}\norm{\rho_h}_{L^q}^{q-2}\Big\} +C\norm{\rho_h}_{L^q}^{q-1} +C\norm{f}_{L^q}^q.
\end{split}
\eeq
By Young's inequality,  
\beq\label{rho:est}
\begin{split}
\frac{d }{dt}\norm{\rho_h}_{L^q}^q &+ \frac{qc_0} 2\int_\Omega |\nabla \rho_h|^{\beta} |\rho_h|^{q-2}  dx\\
&\le   C\norm{\rho_h}_{L^q}^q  + C\Big(1+\norm{\psi}_{L^\infty(\Gamma)}^q+\norm{\psi}_{L^\infty(\Gamma)}^{\frac{q \beta}{2(\beta-1)} } + \norm{f}_{L^q}^q\Big)\\
&\le   C\norm{\rho_h}_{L^q}^q  + C\Big(1+\norm{\psi}_{L^\infty(\Gamma)}^{\frac{q \beta}{2(\beta-1)} } + \norm{f}_{L^q}^q\Big).
\end{split}
\eeq
Solving this differential inequality shows that, 
\beq\label{res0}
\begin{split}
\norm{\rho_h}_{L^q}^q &+\frac{qc_0} 2 \int_0^t\int_\Omega e^{C(t-\tau)} |\nabla \rho_h|^{\beta} |\rho_h|^{q-2}  dxd\tau \\
&\le  \norm{\rho_{h}^0}_{L^q}^q+ C\int_0^t \Big(1+\norm{\psi}_{L^\infty(\Gamma)}^{\frac{q \beta}{2(\beta-1)} } + \norm{f}_{L^q}^q\Big) d\tau.
\end{split}
\eeq
It is easy to see that in both above cases
\beqs
\norm{\rho_h}_{L^q}^q \le  \norm{\rho_{h}^0}_{L^q}^q+ C\int_0^t \Big(1+\norm{\psi}_{L^\infty(\Gamma)}^{\frac{q \beta}{2(\beta-1)} } + \norm{f}_{L^q}^q\Big) d\tau.
\eeqs
Using inequality $ (a+b)^{1/q}\le 2^{1/q}(a^{1/q}+ b^{1/q})$ with 
$$a= \norm{\rho_{h}^0}_{L^q}^q \quad \text{  and } \quad  b= \int_0^t \Big(1+\norm{\psi}_{L^\infty(\Gamma)}^{\frac{q \beta}{2(\beta-1)} } + \norm{f}_{L^q}^q\Big) d\tau,$$  
we obtain 
\beq\label{res1}
\norm{\rho_h}_{L^\infty(I,L^q(\Omega))} \le C\norm{\rho_{h}^0}_{L^q(\Omega)} +C\left[\int_0^T \Big(1+\norm{\psi}_{L^\infty(\Gamma)}^{\frac{q\lambda}{2} } + \norm{f}_{L^q(\Omega)}^q\Big) dt\right]^{\frac 1 q}.
\eeq
Note that  
\beq\label{aaa}
\norm{\rho^0_h}= \norm{\pi\rho^0}\le \norm{\rho^0}.
\eeq 
Thus inequality \eqref{res1a} holds. We finish the proof.  
\end{proof}
 
Next we give estimates in $L^2$-norm following directly from \eqref{res0}  and \eqref{aaa} with $q=2$.   
\begin{lemma} Under assumption of Theorem~\ref{bound-lq}. There exist a positive constant $C$ such that 
\beq\label{resl2}
\begin{aligned}
\norm{\rho_h}_{L^\infty(I,L^2(\Omega))}&+\norm{\nabla \rho_h }_{L^\beta(I,L^\beta(\Omega)) }^{\frac{\beta}2}\\
&\le C\norm{\rho^0}+C\left[\int_0^T \Big(1+\norm{\psi(t)}_{L^\infty(\Gamma)}^{\lambda } + \norm{f(t)}^2\Big) dt\right]^{\frac 12}.
\end{aligned}
\eeq
\end{lemma}
Since the equation \eqref{semidiscreteform} can interpret as the finite system of ordinary differential equations in the coefficients of $\rho_h$ with respect to basis of $W_h$. The stability estimate \eqref{res1} suffices to establish existence of $\rho_h(t)$ for all $t\in I.$   

For the uniqueness of approximation solution, assume that for $i=1,2$, $\rho_h^{(i)} \in W_h$ is the solution of \eqref{semidiscreteform}. Let $\rho_h =\rho_h^{(1)} -\rho_h^{(2)}$ then 
\begin{align*}
(\rho_{h,t}, w_h)+(K(|\nabla \rho_h^{(1)}|)\nabla \rho_h^{(1)}- K(|\nabla \rho_h^{(2)}|)\nabla \rho_h^{(2)},\nabla w_h)=0.
\end{align*}
Choose $w_h = \rho_h,$ implies 
  \beqs
\frac 12 \frac d{dt} \norm{\rho_h}^2+ \intd{K(|\nabla \rho_h^{(1)}|)\nabla \rho_h^{(1)}- K(|\nabla \rho_h^{(2)}|)\nabla \rho_h^{(2)},\nabla  \rho_h^{(1)} - \nabla\rho_h^{(2)}}  =0.
\eeqs
According to \eqref{Mono} we find that 
\begin{multline*}
\intd{K(|\nabla \rho_h^{(1)}|)\nabla \rho_h^{(1)}- K(|\nabla \rho_h^{(2)}|)\nabla \rho_h^{(2)},\nabla  \rho_h^{(1)} - \nabla\rho_h^{(2)} }\\
\ge C\left[1+ \max\{\|\nabla \rho_h^{(1)}\|_{L^{\beta}} ,  \|\nabla \rho_h^{(2)}\|_{L^{\beta}} \}\right]^{-a}\norm{\nabla \rho_h}_{L^{\beta}}^2.
\end{multline*}

Hence 
  \beqs
\frac 12 \frac d{dt} \norm{\rho_h}^2+ C\left[1+ \max\{\|\nabla \rho_h^{(1)}\|_{L^{\beta}} ,  \|\nabla \rho_h^{(2)}\|_{L^{\beta}} \}\right]^{-a}\norm{\nabla \rho_h}_{L^{\beta}}^2  \le 0.
\eeqs
Integrating in time from $0$ to $t$ we have 
\beqs
\norm{\rho_h}^2+ C\int_0^t\left[(1+ \max\{\|\nabla \rho_h^{(1)}\|_{L^{\beta}} ,  \|\nabla \rho_h^{(2)}\|_{L^{\beta}} \}\right]^{-a}\norm{\nabla \rho_h}_{L^{\beta}}^2 dt \le \norm{\rho_h(0)}^2=0.
\eeqs
Thus 
$$
\rho_h=0, \quad \nabla \rho_h = 0  \qquad \forall (x,t)\in\Omega\times I.
$$

Next, we find estimates for $\nabla \rho_h$. We define 
\beq\label{ghkdef} 
g(t) = 1+ \norm{\psi(t)}_{L^\infty(\Gamma)}^\lambda,\quad h(t) = 1+ \norm{\psi_t(t)}_{L^\infty(\Gamma)}^\lambda,  \quad k(t) = g(t)+ \norm{f(t)}^2. 
\eeq
 
\begin{theorem}\label{gradB}  Let $\rho_h$ be a solution to the problem \eqref{semidiscreteform}. Then there exist constant positive constant $C$ satisfying  
\beq\label{mid0}
\norm{ \rho_{h,t}}_{L^2(I,L^2(\Omega))}^2 + \norm{\nabla \rho_h}_{L^{\beta}(I,L^{\beta}(\Omega))}^{\beta}+ \norm{\nabla \rho_h}_{L^\infty(I,L^{\beta}(\Omega))}^{\beta}  +\norm{\rho_h}_{L^\infty(I,L^2(\Omega))}^2\le C \Mm.
\eeq
where 
\beq\label{Mdef}
\begin{split}
\Mm =  \max_{t\in [0,T]} g(t)&+\int_0^T [h(t)+k(t)] dt+ \int_0^T \int_0^t k(\tau) d\tau dt \\
&+\norm{\rho^0}^2+ \norm{\nabla \rho^0}_{L^{\beta}(\Omega)}^{\beta} +\norm{\psi^0}_{L^2(\Gamma)}\norm{\rho^0}_{L^2(\Gamma)}.
\end{split}
\eeq
\end{theorem}
\begin{proof}
Choosing $w_h=\rho_{h,t} $ in \eqref{semidiscreteform} leads to 
\beq\label{Diffineq1}
\begin{split}
\norm{ \rho_{h,t}}^2 +\frac 1 2 \ddt \int_\Omega H(x,t) dx  
&= -\intb{ \psi, \rho_{h,t}  } + (f, \rho_{h,t})\\
&=-\ddt \intb{\psi, \rho_h }+ \intb{ \psi_t, \rho_h }+ (f, \rho_{h,t}).
\end{split}
\eeq
where $H(x,t)=H(|\nabla \rho_h(x,t)| )$ is defined in \eqref{Hdef}.
 
With $q=2$ then from \eqref{rho:est},
\beq\label{rho-est}
\ddt\norm{\rho_h}^2 +c_0\norm{\nabla \rho_h}_{L^{\beta}}^{\beta}\le   C\norm{\rho_h}^2  + C\Big(1+\norm{\psi}_{L^\infty(\Gamma)}^\lambda+ \norm{f}^2\Big).
\eeq
Let
\beqs
\mathcal E(t) = \int_\Omega H(x,t) dx +\norm{\rho_h}^2+2\intb{\psi, \rho_h }.
\eeqs
Adding two inequalities \eqref{Diffineq1} and \eqref{rho-est} gives  
 \beqs
\begin{aligned}
\norm{ \rho_{h,t}}^2  +c_0\norm{\nabla \rho_h}_{L^{\beta}}^{\beta}+ \frac 12\ddt \mathcal E(t) \le \langle \psi_t, \rho_h\rangle+ (f, \rho_{h,t})+ C\norm{\rho_h}^2 + C k(t) .
\end{aligned}
\eeqs
 Using \eqref{bnd-est} and Cauchy's inequality then
  \begin{multline*}
\norm{ \rho_{h,t}}^2 +c_0 \norm{\nabla \rho_h}_{L^{\beta}}^{\beta}+ \frac 12 \ddt \mathcal E(t) \le \frac 12 \left(\norm{\rho_h}^2+ \norm{\nabla \rho_h}_{L^{\beta}}^{\beta}\right)
+ C\left(1+ \norm{\psi_t}_{L^\infty(\Gamma)}^{^\lambda}  \right) \\
+\frac 12 \norm{f}^2 +\frac 12\norm{\rho_{h,t}}^2 + C\norm{\rho_h}^2 + Ck(t) .
\end{multline*}
Absorbing $\norm {f}^2$ to $k(t)$, integrating in time, we obtain
  \beq\label{Diffineq3}
\begin{aligned}
\int_0^T \norm{ \rho_{h,t}}^2 dt &+c_0 \int_0^T \norm{\nabla \rho_h}_{L^{\beta}}^{\beta} dt+\int_\Omega H(x,t) dx +\norm{\rho_h}^2
\le -2\langle \psi, \rho_h\rangle\\
 &+C\int_0^T \norm{\rho_h}^2 dt +C\int_0^T  \left[h(t)  + k(t) \right] dt+ \mathcal E(0).
\end{aligned}
\eeq
Applying \eqref{bnd-eps} to the first term of the left hand side in \eqref{Diffineq3} and using the fact in \eqref{i:ineq4} that  $c_3 (|\nabla\rho_h|^{\beta}-1) \le H(x,t)   \le 2c_2|\nabla\rho_h|^{\beta}$,  we have   
 \beq\label{Diffineq4}
\begin{aligned}
&\int_0^T \norm{ \rho_{h,t}}^2 dt +c_0 \int_0^T\norm{\nabla \rho_h}_{L^{\beta}}^{\beta} dt+  c_3\norm{\nabla \rho_h}_{L^{\beta}}^{\beta} +\norm{\rho_h}^2\\
& \qquad\le 2\varep\left( \norm{\rho_h}^2+ \norm{\nabla \rho_h}_{L^{\beta}}^{\beta}\right)
+ C\left(\varep^{-1}\norm{\psi}_{L^\infty(\Gamma)}+ \varep^{-\frac 1{\beta-1}} \norm{\psi}_{L^\infty(\Gamma)}^{^\lambda}  \right)\\
 &\quad\qquad +C\int_0^T \norm{\rho_h}^2 dt +C\int_0^T  \left[h(t)  + k(t) \right] dt+ \mathcal E(0).
\end{aligned}
\eeq
Thus by taking $\varep =\min\{c_3,1\}/4$, and using Young's inequality, 
 \beq\label{ineq4}
\begin{aligned}
\int_0^T \norm{ \rho_{h,t}}^2 dt &+c_0 \int_0^T \norm{\nabla \rho_h}_{L^{\beta}}^{\beta} dt+ \frac{c_3}2 \norm{\nabla \rho_h}_{L^{\beta}}^{\beta} +\frac 12 \norm{\rho_h}^2 \le C g(t)\\
 &+C\int_0^T \norm{\rho_h}^2 dt +C\int_0^T  \left[h(t)  + k(t) \right] dt+ \mathcal E(0).
\end{aligned}
\eeq
Note from \eqref{res0} with $q=2$ that
\begin{align*}
\int_0^T\norm{\rho_h}^2 dt &\le C\left(T\norm{\rho_h(0)}^2 +\int_0^T \int_0^t k(\tau) d\tau dt \right)\\
&\le C\left(T\norm{\rho^0}^2 +\int_0^T \int_0^t k(\tau) d\tau dt \right)
\end{align*}
and from \eqref{i:ineq4}, \eqref{i:ineq2} that 
\begin{align*}
 \mathcal E(0) \le C\left(\norm{\rho_h^0}^2+ \norm{\nabla \rho_h^0}_{L^{\beta}(\Omega)}^{\beta} +\norm{\psi(0)}_{L^2(\Gamma)} \norm{\rho_h^0 }_{L^2(\Gamma)}\right)\\
 \le C\left(\norm{\rho^0}^2+ \norm{\nabla \rho^0}_{L^{\beta}(\Omega)}^{\beta} +\norm{\psi(0)}_{L^2(\Gamma)} \norm{\rho^0 }_{L^2(\Gamma)}\right).
\end{align*}
The left hand side of \eqref{ineq4} is bounded by $C\Mm$  which implies  \eqref{mid0}.
\end{proof}

Now we prove that the time derivative of density is also bounded. 
\begin{theorem}\label{rhderv} Let $0<t_0<T,$ $\rho_h$ be a solution to the semidiscrete problem \eqref{semidiscreteform}. Then there exist constant positive constant $C$ such that        
\beq\label{mid1}
\norm{\rho_{h,t}}_{L^\infty(I,L^2(\Omega)  )}^2 \le  Ct_0^{-1}\Mm +C\Mm 
 +C\int_0^T h(t)\left(1+\norm{f_t(t)}^2\right)dt.
  \eeq
  where $\Mm$ and $h(t)$ are defined in \eqref{Mdef} and \eqref{ghkdef} respectively. 
\end{theorem}

\begin{proof}
Differentiating \eqref{semidiscreteform} with respect $t$ yields that
 \begin{align*}
(\rho_{h,tt}, w_h) &+ \intd{K(|\nabla \rho_h|)\nabla \rho_{h,t},\nabla w_h} \\
&=- \intd{K'(|\nabla \rho_h|)\frac{\nabla \rho_{h}\cdot \nabla \rho_{h,t}}{|\nabla \rho_h|}\nabla \rho_h ,\nabla w_h } -\intb{ \psi_t, w_h } + \intd{f_t, w_h}.
  \end{align*}
Choosing $ w_h= \rho_{h,t},$ we obtain
    \begin{multline*}
  \frac{1}{2}\ddt \norm{\rho_{h,t}}^2 +\norm{K^{1/2}(|\nabla \rho_h|)\nabla \rho_{h,t}}^2\\
   =- \intd{K'(|\nabla \rho_h|)\frac{\nabla \rho_{h}\cdot \nabla \rho_{h,t}}{|\nabla \rho_h|}\nabla \rho_h ,\nabla \rho_{h,t}   } + \intd{ f_t, \rho_{h,t} }-\intb{\psi_t, \rho_{h,t}}.
  \end{multline*}
  Using \eqref{i:ineq3}, 
  \beq\label{I1}
  \left|- \intd{K'(|\nabla \rho_h|)\frac{\nabla \rho_{h}\cdot \nabla \rho_{h,t}}{|\nabla \rho_h|}\nabla \rho_h ,\nabla \rho_{h,t}   }\right|\le a\norm{K^{1/2}(|\nabla \rho_h|)\nabla \rho_{h,t}}^2.
  \eeq
  Thus 
  \beq\label{midstep0}
  \frac{1}{2}\ddt \norm{\rho_{h,t}}^2 +(1-a)\norm{K(|\nabla \rho_h|)\nabla \rho_{h,t}}^2\le \intd{ f_t, \rho_{h,t} }-\intb{\psi_t, \rho_{h,t}}.
  \eeq
 In virtue of Cauchy's inequality, for all $\varep>0$ 
\beq\label{I2}
( f_t,\rho_{h,t} ) \le \varep\norm{\rho_{h,t}}^2 + C\varep^{-1}\norm{f_t}^2.
\eeq
Using Trace Theorem we obtain,
  \beq\label{I30}
  \left|\intb{\psi_t, \rho_{h,t}}\right|\le \norm{\psi_t}_{L^\infty(\Gamma)}\left[ \intd{|\rho_{h,t}| ,1}+\intd{|\nabla \rho_{h,t}|,1 }\right].
  \eeq
  Again Cauchy's inequality gives that for all $\varep,\varep_1>0$
   \beq\label{Yt2}
   (|\rho_{h,t}|,1) \le \varep\norm{\rho_{h,t}}^2+C\varep^{-1}, 
\eeq
and
\beqs
 (|\nabla \rho_{h,t}|,1) 
 \le  \varep_1 ( K(|\nabla \rho_h|)|\nabla \rho_{h,t}|^2,1)+C\varep_1^{-1} ( K^{-1}(|\nabla \rho_h|,1). 
 \eeqs 
By using \eqref{i:ineq1} and $(1+x)^a \le 1+x^a, x\ge 0$ imply   
 \beq\label{Yt1}
 \begin{aligned}
 (|\nabla \rho_{h,t}|,1) &\le \varep_1 \intd{K(|\nabla \rho_h|)|\nabla \rho_{h,t}|^2,1} + C\varep_1^{-1} (1+|\nabla \rho_h|^a,1)  \\ 
  &\le  \varep_1 \norm{K^{1/2}(|\nabla \rho_h|)\nabla \rho_{h,t}}^2+C\varep_1^{-1}(1+\norm{\nabla \rho_h}_{L^{\beta}}^{a}). 
 \end{aligned} 
 \eeq
 
 Combining \eqref{I30}, \eqref{Yt2} and \eqref{Yt1} shows that  
 \beq\label{I3}
 \begin{split}
\left|\intb{\psi_t, \rho_{h,t}}\right|\le\norm{\psi_t}_{L^\infty(\Gamma)} \Big\{&\varep\norm{\rho_{h,t}}^2+C\varep^{-1}\\
 &+ \varep_1\norm{ K^{1/2}(|\nabla \rho_h|)\nabla \rho_{h,t} }^2+ C\varep_1^{-1}(1+\norm{\nabla \rho_h}_{L^{\beta}}^{a})  \Big\}.
\end{split}
\eeq
 It follows from \eqref{midstep0}, \eqref{I2} and \eqref{I3} that  
\beq\label{bb}
\begin{split}
&  \frac{1}{2}\ddt \norm{\rho_{h,t}}^2 +(1-a)\norm{K^{1/2}(|\nabla \rho_h|)\nabla \rho_{h,t}}^2  \le \norm{\psi_t}_{L^\infty(\Gamma)} \Big\{\varep\norm{\rho_{h,t}}^2+C\varep^{-1}\\
  & + \varep_1\norm{ K^{1/2}(|\nabla \rho_h|)\nabla \rho_{h,t} }^2 +C\varep_1^{-1}(1+ \norm{\nabla\rho_h}_{L^{\beta}}^{a})  \Big\}+\varep\norm{\rho_{h,t}}^2 +C\varep^{-1}\norm{f_t}^2.
  \end{split}
 \eeq
 
   If $\norm{\psi_t(t)}_{L^\infty(\Gamma)}=0$ then 
  \beqs
  \frac{1}{2}\ddt \norm{\rho_{h,t}}^2 +(1-a)\norm{K^{1/2}(|\nabla \rho_h|)\nabla \rho_{h,t}}^2 \le \varep\norm{\rho_{h,t}}^2 +C\varep^{-1}\norm{f_t}^2.
  \eeqs
  
  Dropping the second term of the left hand side in above estimate, selecting $\varep=1/2$  and, for $0<t'\le t_0<t\le T$ integrating from $t'$ to $t$ give   
   \beq
   \begin{split}
  \norm{\rho_{h,t}(t)}^2 d\tau &\le \norm{\rho_{h,t}(t')}^2 +\int_{t'}^t \norm{\rho_{h,t}}^2 d\tau +C\int_{t'}^t \norm{f_t}^2d\tau\\
  &\le\norm{\rho_{h,t}(t')}^2 +\int_0^t \norm{\rho_{h,t}}^2 d\tau +C\int_0^t \norm{f_t}^2d\tau.
  \end{split}
  \eeq 
Now integrating in $t'$ from $0$ to $t_0$, 
  \beq
   t_0\norm{\rho_{h,t}(t)}^2 d\tau \le \int_0^{t_0} \norm{\rho_{h,t}(t')}^2 +t_0\Big(\int_0^t \norm{\rho_{h,t}}^2 d\tau +C\int_0^t \norm{f_t}^2d\tau\Big).
 \eeq 
  Using \eqref{mid0}, we find that
  \beq
   t_0\norm{\rho_{h,t}(t)}^2 d\tau \le \Mm +Ct_0\Big(\Mm+\int_0^t \norm{f_t}^2d\tau\Big),
 \eeq 
  which holds \eqref{mid1}.
  
  Consider $\norm{\psi_t}_{L^\infty(\Gamma)}\neq 0$.  Selecting $\varep_1=\frac{1-a} 2\norm{\psi_t}_{L^\infty(\Gamma)}^{-1}$  and $\varep =\frac 12(\norm{\psi_t}_{L^\infty(\Gamma)}+1 )^{-1} $ in \eqref{bb} yields   
    \begin{multline*}
\ddt \norm{\rho_{h,t}}^2 +(1-a)\norm{ K(|\nabla\rho_h|)\nabla \rho_{h,t} }^2\le \norm{\rho_{h,t}}^2 +C \norm{\psi_t}_{L^\infty(\Gamma)}^2\norm{\nabla \rho_h}_{L^{\beta}}^{a}\\
\quad+C\left(\norm{\psi_t}_{L^\infty(\Gamma)}+1 \right) \left(\norm{\psi_t}_{L^\infty(\Gamma)} +\norm{f_t}^2\right)
 +C\norm{\psi_t}_{L^\infty(\Gamma)}^2\\
\le \norm{\rho_{h,t}}^2 +\norm{\nabla \rho_h}_{L^{\beta}}^{\beta}+C Z(t),
\end{multline*}
where 
\beqs
Z(t) =\left(\norm{\psi_t}_{L^\infty(\Gamma)}+1 \right) \left(\norm{\psi_t}_{L^\infty(\Gamma)} +\norm{f_t}^2\right)
 +\norm{\psi_t}_{L^\infty(\Gamma)}^2+ \norm{\psi_t}_{L^\infty(\Gamma)}^\lambda.
\eeqs

  For $t\ge t_0 \ge t'>0$. Ignoring the the nonnegative term in the left hand side of above inequality, integrating from $t'$ to $t$ and then integrating in $t'$ from $0$ to $t_0$, we find that    
\beqs
 t_0 \norm{\rho_{h,t}}^2  \le  \int_0^{t_0}\norm{\rho_{h,t}(t')}^2 dt'+t_0\int_0^t \Big[\norm{\rho_{h,t}}^2 +\norm{\nabla \rho_h}_{L^{\beta}}^{\beta}\Big]d\tau
 +Ct_0\int_{0}^t Z(\tau)d\tau.
  \eeqs
By virtue of \eqref{mid0},   
\beqs
\int_{0}^t\Big[\norm{\rho_{h,t}}^2+\norm{\nabla \rho_h}_{L^{\beta}}^{\beta}\Big]dt\le C\Mm,\quad \int_0^{t_0}\norm{\rho_{h,t}(t')}^2 dt' \le C\Mm.
\eeqs    
Therefore 
\beq\label{al1}
 t_0 \norm{\rho_{h,t}}^2 \le  C\Mm +C\Mm t_0
 +Ct_0\int_{0}^t Z(\tau) d\tau.
  \eeq
We estimate $Z$-term by
\beq\label{al2}
\begin{split}
Z(t)&\le \left(\norm{\psi_t}_{L^\infty(\Gamma)}+1 \right) \norm{f_t}^2
 +C\left(1+ \norm{\psi_t}_{L^\infty(\Gamma)}^\lambda\right)\\
 &\le C\left(1+ \norm{\psi_t}_{L^\infty(\Gamma)}^\lambda\right)\left(1+\norm{f_t}^2\right).
\end{split}
\eeq
The inequality \eqref{mid1} follows from \eqref{al1} and \eqref{al2}. The proof is complete.  
\end{proof}


The following results can be proved by following the ideas of the proof of Theorem \ref{bound-lq}, \ref{gradB} and \ref{rhderv}.     

\begin{theorem}  Let $0<t_0<T,$ and $q\ge 2$, $\rho$ be a solution to problem \eqref{weakform}. Then there exist constant positive constant $C$ such that
\beq\label{res4r}
\norm{\rho}_{L^\infty(I,L^q(\Omega))} \le C\norm{\rho^0}_{L^q(\Omega)} +C\Big(\int_0^T \big[ 1+\norm{\psi}_{L^\infty(\Gamma)}^\frac{q\lambda}{2} +\norm{f(t)}_{L^q(\Omega)}^q\big]dt\Big)^{\frac 1 q},
\eeq
\beq\label{mid4gradr}
\norm{ \rho_t}_{L^2(I,L^2(\Omega))}^2 + \norm{\nabla \rho}_{L^{\beta}(I,L^{\beta}(\Omega))}^{\beta}+ \norm{\nabla \rho}_{L^\infty(I,L^{\beta}(\Omega))}^{\beta}  +\norm{\rho_h}_{L^\infty(I,L^2(\Omega))}^2\le C \Mm,
\eeq
\beq\label{mid4r-t}
\norm{\rho_t}_{L^\infty(I,L^2(\Omega)  )}^2 \le  Ct_0^{-1}\Mm +C\Mm 
 +C\int_0^T h(t)\left(1+\norm{f_t(t)}^2\right)dt,
  \eeq
where $h$ and $\Mm$ are defined as in \eqref{ghkdef} and \eqref{Mdef}. 
\end{theorem}

\section{Error estimates}\label{errSec} In this section we will establish the error estimates between analytical solution and approximation solution in several norms.

\subsection{Error estimate for continuous Galerkin method} We will find the error bounds in the semidiscrete method by comparing the computed solution to the projections of the true solutions. To do this, we restrict the test functions in \eqref{weakform} to the finite dimensional spaces. Let
\beq
\chi=\rho -\rho_h = (\rho -\pi\rho) - (\rho_h- \pi\rho) \eqdef \vartheta -\theta_h.
\eeq
\begin{theorem}\label{err-L2}   Let $\rho,\rho_h$ be two solution to problems \eqref{weakform} and \eqref{semidiscreteform} respectively. Assume that $\rho\in L^\infty(I,H^r(\Omega)\cap W^{r,\beta}(\Omega)  )$, $\rho_t\in L^2(I,H^r(\Omega) )$. Then there exist positive constants $C$ independence of $h$ such that    
  \beq\label{err20}
  \begin{split}
 \norm {\rho -\rho_h}_{L^\infty (I,L^2(\Omega))}&+ \mathcal M^{-\frac\gamma 2}\norm{\nabla(\rho -\rho_h)}_{L^2(I,L^\beta(\Omega) ) } \le Ch^r\norm{\rho}_{L^\infty (I,H^r(\Omega))}\\
 &+ C\Mm^{\frac{1}{2\beta\lambda}} h^{\frac{r-1}2 } \norm{\rho}_{L^1(I,W^{r,\beta}(\Omega))}^{\frac 1 2}+CT^{\frac12}h^r\norm{\rho_t}_{L^2(I,H^r(\Omega))},
 \end{split}
 \eeq
  where $\mathcal M$ is defined as in \eqref{Mdef}.
\end{theorem}
\begin{proof}
Subtracting \eqref{semidiscreteform} from \eqref{weakform} we obtain the error equation 
\beq\label{errEq}
 (\rho_{t}- \rho_{h,t}, w_h) + (K(|\nabla \rho|)\nabla \rho -K(|\nabla \rho_h|)\nabla \rho_h ,\nabla w_h) =0,  \quad w_h\in W_h.
\eeq
Choosing $w_h=\theta_h$ in \eqref {errEq} gives
\beq\label{Erreql2}
\begin{aligned}
 \frac 1 2\ddt \norm{\theta_h}^2 &+(K(|\nabla \rho|)\nabla \rho -K(|\nabla \rho_h|)\nabla \rho_h ,\nabla\rho -\nabla\rho_h ) \\
 &= (\vartheta_t,\theta_h) +  (K(|\nabla \rho|)\nabla \rho -K(|\nabla \rho_h|)\nabla \rho_h ,\nabla \vartheta). 
 \end{aligned} 
 \eeq
By the monotonicity of $K(\cdot)$ in \eqref{Mono}, there is a positive constant $C_0$ independence of $h$ and $t$ satisfying
 \beqs 
 (K(|\nabla \rho|)\nabla \rho -K(|\nabla \rho_h|)\nabla \rho_h ,\nabla\rho -\nabla\rho_h)  \ge C_0 \omega(t)\norm{\nabla\chi}_{L^\beta}^2. 
    \eeqs
    where 
    $   \omega(t) = \left(1+\norm{\nabla \rho}_{L^{\beta}}+ \norm{\nabla \rho_h}_{L^{\beta}}\right)^{-a}. $
    
Observing  from \eqref{mid0} that there is a positive constant $C_1$ independence of $h$ and $t$ such that 
 \beq\label{Odef}
 \begin{split}
 \omega^{-1}(t)&=\left(1+\norm{\nabla \rho}_{L^{\beta}}+ \norm{\nabla \rho_h}_{L^{\beta}}\right)^{\beta\gamma}\\
 & \le 2^\beta \left(1+\norm{\nabla \rho}_{L^{\beta}}^\beta+ \norm{\nabla \rho_h}_{L^{\beta}}^\beta\right)^{\gamma} \le C_1\mathcal M^{\gamma}. 
 \end{split}
 \eeq
Thus 
\beq\label{r0} 
 (K(|\nabla \rho|)\nabla \rho -K(|\nabla \rho_h|)\nabla \rho_h ,\nabla\rho -\nabla\rho_h)  \ge \frac{C_0}{C_1} \Mm^{-\gamma} \norm{\nabla\chi}_{L^\beta}^2. 
\eeq                      
By Cauchy's inequality, for $\varep_0>0$  
\beq\label{r1}
(\vartheta_t, \theta_h) \le C\varep_0^{-1} \norm{\vartheta_t}^2 +\varep_0 \norm{\theta_h}^2,
\eeq
Using \eqref{i:ineq2}, H\"older inequality and then \eqref{mid0} we obtain
\beq\label{r2}
\begin{aligned}
(K(|\nabla \rho|)\nabla \rho -K(|\nabla \rho_h|)\nabla \rho_h ,\nabla \vartheta)
&\le C(|\nabla \rho|^{\beta-1}+|\nabla \rho_h|^{\beta-1} ,|\nabla \vartheta|)\\
& \le C(\norm{\nabla \rho}_{L^\beta}^{\frac{1}\lambda}+\norm{\nabla \rho_h}_{L^\beta}^{\frac{1}\lambda})\norm{\nabla \vartheta}_{L^\beta} \\
&\le C (\norm{\nabla \rho}_{L^\beta}+\norm{\nabla \rho_h}_{L^\beta})^{\frac{1}{\lambda}}\norm{\nabla \vartheta}_{L^\beta}\\
&\le C\Mm^{\frac{1}{\beta\lambda}}\norm{\nabla \vartheta}_{L^\beta}.
\end{aligned}
\eeq
We obtain from \eqref{Erreql2}, \eqref{r0}, \eqref{r1} and \eqref{r2} that 
\begin{align*}
\frac 12\ddt\norm{\theta_h}^2+\frac{C_0}{C_1} \mathcal M^{-\gamma} \norm{\nabla\chi}_{L^\beta}^2\le C\Mm^{\frac{1}{\beta\lambda}}\norm{\nabla \vartheta}_{L^\beta}+C\varep_0^{-1}\norm{\vartheta_t}^2 +\varep_0\norm{\theta_h}^2.
\end{align*}
Integrating in time from $0$ to $T$ and then taking sup-norm in time of resultant show that 
\beqs
\begin{split}
\norm{\theta_h}_{L^\infty(I,L^2)}^2+ \mathcal M^{-\gamma} \int_0^T \norm{\nabla\chi}_{L^\beta}^2 dt \le C\Mm^{\frac{1}{\beta\lambda}}\int_0^T \norm{\nabla \vartheta}_{L^\beta} dt+C\varep_0^{-1}\int_0^T \norm{\vartheta_t}^2 dt+\varep_0 T \norm{\theta_h}_{L^\infty(I,L^2)}^2.
\end{split}
\eeqs
Selecting $\varep_0 =\frac 1{2T},$ we find that 
\beq\label{errl2}
\norm{\theta_h}_{L^\infty(I,L^2)}^2+\Mm^{-\gamma} \int_0^T \norm{\nabla\chi}_{L^\beta}^2 dt \le C\Mm^{\frac{1}{\beta\lambda}}\int_0^T \norm{\nabla \vartheta}_{L^\beta} dt+CT\int_0^T \norm{\vartheta_t}^2 dt.
\eeq
It follows directly from \eqref{errl2} that  
\beq\label{core0}
\begin{split}
\norm{\theta_h}_{L^\infty(I,L^2)}  + \mathcal M^{-\frac\gamma 2}\norm{\nabla\chi}_{L^2(I,L^\beta)} &\le C\Mm^{\frac{1}{2\beta\lambda}}\norm{\nabla \vartheta}_{L^1(I,L^\beta)}^{\frac 1 2}+CT^{\frac 1 2}\norm{\vartheta_t}_{L^2(I,L^2)}\\
&\le C\Mm^{\frac{1}{2\beta\lambda}} h^{\frac{r-1} 2}\norm{\rho}_{L^1(I,W^{r,\beta})}^{\frac 1 2}+CT^{\frac12}h^r \norm{\rho_t}_{L^2(I,H^r)}.
\end{split}
\eeq
By triangle inequality, 
  \beqs
  \norm{\chi}_{L^\infty (I,L^2)} \le \norm{\vartheta}_{L^\infty (I,L^2)}+\norm{\theta_h}_{L^\infty (I,L^2)}. 
   \eeqs
This and \eqref{core0} imply \eqref{err20}. Thus the proof is complete.  
 \end{proof}

Now we give the error estimate in $L^q$-norm for any $q>2.$
\begin{theorem}\label{err-Lalp}  Let $q\in(2, \infty)$, $\rho$ solve the problem  \eqref{weakform} and $\rho_h$ solve the semidiscrete problem \eqref{semidiscreteform} . Assume that $\rho\in L^\infty(I,W^{r,q}(\Omega)\cap W^{r,2q}(\Omega))$, $\nabla\rho\in L^{2q}(I,L^{2q}(\Omega))$ , $\rho_t\in L^q(I, W^{r,q} (\Omega))$. Then there exist constant positive constant $C$ independence of $h$ and $T$ such that   
\beq \label{err1}
\begin{split}
 \norm {\rho -\rho_h}_{L^\infty (I,L^q(\Omega))}\le \norm{\rho -\pi\rho}_{L^\infty (I,L^q(\Omega))}&+ CT^{1-\frac 1 q} \norm{(\rho -\pi\rho)_t}_{L^q(I,L^q(\Omega))}\\
 &+ C\mathcal A^{\frac1q} T^{\frac12-\frac1q}\norm{\nabla (\rho -\pi\rho)}_{L^{2q}(I,L^{2q}(\Omega))},
 \end{split}
 \eeq
 where 
 \beq\label{Adef}
\mathcal A= \left(\int_0^T \left[1+\norm{\nabla \rho}_{L^{aq}(\Omega)}^{aq}+ \norm{\nabla \rho_h}_{L^{aq}(\Omega) }^{aq} \right]dt\right)^{\frac 1 2}.
\eeq
Consequently, 
\beq\label{err2}
  \begin{split}
 \norm {\rho -\rho_h}_{L^\infty (I,L^q(\Omega))}\le Ch^r\norm{\rho}_{L^\infty (I,W^{r,q}(\Omega))}&+ CT^{1-\frac 1 q}h^r \norm{\rho_t}_{L^q(I,W^{r,q}(\Omega))}\\
&+ C\mathcal A^{\frac 1q} T^{\frac12-\frac1q} h^{r-1}\norm{\rho }_{L^{2q}(I,W^{r,2q}(\Omega))}.
 \end{split}
 \eeq
\end{theorem}
\begin{proof}
Choosing $w_h=|\theta_h|^{q-1}\sign \theta_h$ in \eqref {errEq} we have
\beqs
 (\vartheta_t -\theta_{h,t}, |\theta_h|^{q-1}\sign \theta_h ) +(q-1) (K(|\nabla \rho|)\nabla \rho -K(|\nabla \rho_h|)\nabla \rho_h , |\theta_h|^{q-2}\nabla \theta_h ) =0.
\eeqs
We rewrite as form
\begin{align*}
 (\theta_{h,t}, |\theta_h|^{q-1}\sign \theta_h) +(q-1)(K(|\nabla \rho|)\nabla \rho -K(|\nabla \rho_h|)\nabla \rho_h ,|\theta_h|^{q-2}(\nabla\rho -\nabla\rho_h) ) \\
 = (\vartheta_t,|\theta_h|^{q-1}\sign \theta_h) +  (q-1)(K(|\nabla \rho|)\nabla \rho -K(|\nabla \rho_h|)\nabla \rho_h ,|\theta_h|^{q-2}\nabla \vartheta). 
 \end{align*} 
 The monotonicity of $K(\cdot)$ in \eqref{Qineq} provides
 \beq 
 (K(|\nabla \rho|)\nabla \rho -K(|\nabla \rho_h|)\nabla \rho_h ,|\theta_h|^{q-2}(\nabla\rho -\nabla\rho_h))  \ge (\beta-1) (K(\xi)|\nabla\chi|^2, |\theta_h|^{q-2} ), 
    \eeq
    where $\xi=\max\{|\nabla \rho|, |\nabla \rho_h|  \}.$
Hence 
\begin{multline*}
 \frac 1 q \ddt\norm{\theta_h}_{L^q}^q+ (q-1)(\beta-1)(K(\xi)|\nabla \chi|^2,|\theta_h|^{q-2}) \\
  \le (\vartheta_t, |\theta_h|^{q-1}\sign \theta_h) +  (q-1)(K(|\nabla \rho|)\nabla \rho -K(|\nabla \rho_h|)\nabla \rho_h ,|\theta_h|^{q-2}\nabla \vartheta).
\end{multline*} 
Using Young's inequality with exponents $q$ and $\frac{q}{q-1}$,  
\beq\label{r1a}
(\vartheta_t, |\theta_h|^{q-1}\sign \theta_h) \le q^{-1}\left(\frac{q}{q-1}\right)^{1-q} \varep^{1-q}\norm{\vartheta_t}_{L^q}^{q} +\varep \norm{\theta_h}_{L^q}^{q}.
\eeq
Using \eqref{Kcont}, H\"older's inequality and Young's inequality we obtain, for $\varep, \varep_0>0,$
\beq\label{r2a}
\begin{aligned}
&(K(|\nabla \rho|)\nabla \rho -K(|\nabla \rho_h|)\nabla \rho_h ,|\theta_h|^{q-2}\nabla \vartheta)
\le C (|\nabla\chi|,|\theta_h|^{q-2}|\nabla \vartheta|)\\
&\qquad\le \varep_0 (K(\xi)|\nabla \chi|^2,|\theta_h|^{q-2}) +C\varep_0^{-1}(K^{-1}(\xi)|\nabla \vartheta|^2 ,|\theta_h|^{q-2})\\
&\qquad\le \varep_0 (K(\xi)|\nabla \chi|^2,|\theta_h|^{q-2})+ \varep\norm{\theta_h}_{L^q}^q + C\varep_0^{-\frac q 2}\varep^{-\frac{q-2}2}\norm{K^{-\frac 1 2}(\xi)\nabla \vartheta}_{L^q}^{q}.
\end{aligned}
\eeq
From \eqref{r1a}, \eqref{r2a} with $\varep_0=(q-1)(\beta-1)/2$,  we find that 
\begin{multline}\label{ineq2}
\frac 1 q \ddt\norm{\theta_h}_{L^q}^q+\frac{(q-1)(\beta-1)}2 (K(\xi)|\nabla \chi|^2,|\theta_h|^{q-2})\\
 \le C\varep^{1-q} \norm{\vartheta_t}_{L^q}^{q} + 2\varep\norm{\theta_h}_{L^q}^{q} +C\varep^{-\frac{q-2}2}\norm{K^{-\frac 1 2}(\xi)\nabla \vartheta}_{L^q}^{q}.
\end{multline}
Integrating in time from $0$ to $T$,
\begin{align*}
&\frac 1 q \norm {\theta_h}_{L^\infty(I,L^q)}^q +\frac {(q-1)(\beta-1)}2\int_0^T (K(\xi)|\nabla \chi|^2,|\theta_h|^{q-2}) dt\\
&\qquad\le  2\varep T\norm{\theta_h}_{L^\infty(I,L^q)}^{q} +C\varep^{1-q}\int_0^T \norm{\vartheta_t}_{L^q}^{q} dt +C\varep^{-\frac{q-2}2}\int_0^T \norm{K^{-\frac 1 2}(\xi)\nabla \vartheta}_{L^q}^{q}dt.
\end{align*}    
Choosing $\varep = \frac 1{4qT}$, using $\eqref{i:ineq1}$ we find that
\begin{align*}
\norm {\theta_h}_{L^\infty(I,L^q)}^q &+ \int_0^T (K(\xi)|\nabla \chi|^2,|\theta_h|^{q-2}) ds\\
 &\le CT^{q-1}\int_0^T \norm{\vartheta_t}_{L^q}^{q} dt +CT^{\frac{q-2}2}\int_0^T \norm{K^{-\frac 1 2}(\xi)\nabla \vartheta}_{L^q}^{q}dt.
\end{align*}   
Note that 
\begin{align*}
\int_0^T \norm{K^{-\frac 1 2}(\xi)\nabla \vartheta}_{L^q}^{q}dt &= \int_\Omega \int_0^T K^{-\frac q 2}(\xi)|\nabla \vartheta|^{q}dt dx\\
&\le \int_\Omega \Big(\int_0^T K^{-q}(\xi)dt\Big)^{\frac 1 2} \Big(\int_0^T|\nabla \vartheta|^{2q}dt\Big)^{\frac 1 2} dx\\
&\le  \Big(\int_0^T \int_\Omega K^{-q}(\xi)dxdt\Big)^{\frac 1 2} \Big(\int_0^T\int_\Omega|\nabla \vartheta|^{2q}dxdt\Big)^{1/2}\\
&\le  C\Big(\int_0^T \int_\Omega (1+\xi^{aq})dxdt\Big)^{\frac 1 2} \norm{ \nabla \vartheta}_{L^{2q}(I,L^{2q}) }^q.
\end{align*}
Moreover $1+\xi^{aq} \le C(1+ \max\{|\nabla \rho|^{aq}, |\nabla \rho_h|^{aq} \}). $ 
 Thus
\begin{multline*}
\norm {\theta_h}_{L^\infty(I,L^q)}^q + \int_0^T (K(\xi)|\nabla \chi|^2,|\theta_h|^{q-2}) dt  \le CT^{q-1}\int_0^T \norm{\vartheta_t}_{L^q}^{q} dt\\ +CT^{\frac{q-2}2}\Big(\int_0^T \int_\Omega 1+ \max\{|\nabla \rho|^{aq}, |\nabla \rho_h|^{aq} \}dxdt\Big)^{\frac 1 2}    \norm{ \nabla \vartheta}_{L^{2q}(I,L^{2q}) }^q.
\end{multline*}   
Dropping the second  term of the LHS in above estimate  we obtain  
\beq\label{sp2}
\norm {\theta_h}_{L^\infty(I,L^q)} \le CT^{1-\frac 1 q} \norm{\vartheta_t}_{L^q(I,L^q)}+ C\mathcal A^{\frac 1q} T^{\frac12-\frac 1q}\norm{\nabla \vartheta}_{L^{2q}(I,L^{2q})}.
\eeq
Inequality \eqref{err1} follows thanks to \eqref{sp2} and triangle inequality 
  \beqs
  \norm{\chi}_{L^\infty (I,L^q)} \le \norm{\vartheta}_{L^\infty (I,L^q)}+\norm{\theta_h}_{L^\infty (I,L^q)}. 
   \eeqs
 This finishes the proof.  
 \end{proof}

For the completeness we derive $L^\infty$-estimate  for $\rho -\rho_h$. 

\begin{theorem}  Let $\rho,\rho_h$ be solutions to problems \eqref{weakform} and \eqref{semidiscreteform} respectively.  Assume that $\rho\in L^\infty(I,W^{r,\infty}(\Omega) )$, 
$\rho_t\in L^2(I,W^{r,q}(\Omega) )$. Then there exist constant positive constant $C$ independence of $h$ such that    
\beq\label{err3}
\begin{split}
 \norm {\rho -\rho_h}_{L^\infty (I,L^\infty(\Omega))}\le Ch^r\norm{\rho}_{L^\infty (I,W^{r,\infty}(\Omega))} &+ CT^{1-\frac 1 q}h^{r-\frac d q} \norm{\rho_t}_{L^q(I,W^{r,q}(\Omega))}\\
 &+ C\mathcal A^{\frac 1q} T^{\frac12-\frac1q} h^{r-1-\frac d q}\norm{\rho }_{L^{2q}(I,W^{r,2q}(\Omega))}.
 \end{split}
 \eeq
 \end{theorem}
\begin{proof} 
Recall that the quasi-uniform of $\mathcal T_h$ we have the inverse estimate (see \cite{BS08,TV06})
\beqs
\norm{\theta_h}_{L^\infty(\Omega)}\le Ch^{-\frac d  q}\norm{\theta_h}_{L^q(\Omega)}.
\eeqs
This and triangle inequality imply that 
\begin{align*} 
   \norm{\chi}_{L^\infty(I,L^\infty)} 
  &\le   \norm{\vartheta}_{L^\infty(I,L^\infty)}+\norm{\theta_h}_{L^\infty(I,L^\infty)} \\
  &\le  Ch^{r}\norm{\rho}_{L^\infty(I,W^{r,\infty})}+Ch^{-\frac d q}\norm{\theta_h}_{L^\infty(I,L^q)}.
 \end{align*} 
 Thus  inequality \eqref{err3} follows directly from\eqref{sp2}.   
\end{proof}

\subsection { Error estimate for gradient}

Now we give an error estimate for gradient
\begin{theorem}Under the assumption of Theorem~\ref{err-L2}. For any $0<t_0\le t\le T$. There exist constant positive constants $C$ independence of $h$ such that    
\beq\label{ErrGrad}
\begin{aligned}
 \norm {\nabla (\rho -\rho_h)(t)}_{L^\beta(\Omega) }  &\le C \Dd \Mm^{\frac \gamma 2} h^{\frac{r-1}4} \Big\{\Mm^{\frac 1{4\beta\lambda} } \norm{\rho(t)}_{L^2(I,W^{r,\beta}(\Omega))}^{\frac 1 4}\\
&\quad +T^{\frac14}h^{\frac {r+1} 4} \norm{\rho_t(t)}_{L^2(I,H^r(\Omega))}^{\frac 1 2}\Big\} +C \mathcal M^{\frac 1{2\beta\lambda}}h^{\frac{r-1} 2}\norm{\rho(t)}_{W^{r,\beta}(\Omega)}^{\frac 12}.
\end{aligned}
\eeq
where the positive constantt $\Mm$ is defined as in \eqref{Mdef}, and  
\beq\label{Ddef}
\Dd = \Big[ (t_0^{-1}+1)\Mm 
 +\int_0^T h(t)\left(1+\norm{f_t(t)}^2\right)dt\Big]^{\frac 14}. 
\eeq
\end{theorem}
\begin{proof}
Choosing  $w_h=\theta_h$ in \eqref {errEq} we have
\begin{multline}\label{sq1}
 \intd{K(|\nabla \rho|)\nabla \rho -K(|\nabla \rho_h|)\nabla \rho_h ,\nabla \rho -\nabla\rho_h }
 \\
  =(\rho_t -\rho_{h,t}, \theta_h) + \intd{K(|\nabla \rho|)\nabla \rho -K(|\nabla \rho_h|)\nabla \rho_h ,\nabla \vartheta}.
\end{multline}
By \eqref{Mono} and \eqref{Odef}, there is a positive constant $C_0$ independence of $h$ such that 
\beqs
(K(|\nabla \rho|)\nabla \rho -K(|\nabla \rho_h|)\nabla \rho_h ,\nabla \rho -\nabla\rho_h )\ge C_0\mathcal M^{-\gamma} \norm{\nabla (\rho-\rho_h)}_{L^\beta}^2.
\eeqs
In virtue of triangle inequality, Holder's inequality we have      
\beq\label{ineq0}
\begin{split}
&(\rho_t -\rho_{h,t}, \theta_h) + \intd{K(|\nabla \rho|)\nabla \rho -K(|\nabla \rho_h|)\nabla \rho_h ,\nabla \vartheta}\\
&\qquad\qquad\qquad\le  (|\rho_t|+|\rho_{h,t}|,|\theta_h|) + C_1(|\nabla \rho|^{\beta-1}+ |\nabla \rho_h|^{\beta-1},|\nabla \vartheta|)\\
&\qquad\qquad\qquad\le (\norm{\rho_t}+\norm{\rho_{h,t}} )\norm{\theta_h} +C_1\Mm ^{\frac 1{\beta\lambda}} \norm{\nabla \vartheta}_{L^\beta}
\end{split}
\eeq
for some positive constant $C_1$ independence of $h$. 

Thus 
\begin{align*}
 \norm {\nabla (\rho -\rho_h)}_{L^\beta}^2\le C_0^{-1}\mathcal M^{\gamma}(\norm{\rho_t}+\norm{\rho_{h,t}} )\norm{\theta_h} + C_1C_0^{-1} \Mm^{\frac{1}{\beta\lambda}}\norm{\nabla \vartheta}_{L^\beta}.
\end{align*}
Thanks to \eqref{mid1} and \eqref{mid4r-t},
\beq
\norm{\rho_t}+\norm{\rho_{h,t}} \le 2\mathcal D^2.
\eeq
Hence there is a positive constant $C$ independence of $h$ such that 
\beq\label{errgrad}
 \norm {\nabla (\rho -\rho_h)}_{L^\beta}\le C \Dd \Mm^{\frac {\gamma}{2}}  \norm{\theta_h}^{\frac 1 2 } +C \Mm^{\frac{1}{2\beta\lambda}}\norm{\nabla \vartheta}_{L^\beta}^{\frac 1 2}.
\eeq
Due to \eqref{core0} and the fact that $\norm{\nabla \vartheta}_{L^\beta}\le Ch^{r-1}\norm{\rho}_{W^{r,\beta}}$ the left hand side of \eqref{errgrad} is bounded by  
\beq\label{lhs0}
C\Dd \Mm^{\frac {\gamma}{2} } \left[\Mm^{\frac 1 {2\beta\lambda} }h^{\frac{r-1} 2}\norm{\rho}_{L^2(I,W^{r,\beta})}^{\frac 12}+T^{\frac 1 2}h^r \norm{\rho_t}_{L^2(I,H^r)}\right]^{\frac 1 2}+ Ch^{\frac{r-1}2}\Mm ^{\frac{1}{2\beta\lambda} }\norm{\rho}_{W^{r,\beta}}^{\frac 1 2}.
\eeq
Thus \eqref{ErrGrad} follows from \eqref{errgrad} and \eqref{lhs0}.
\end{proof}

\subsection{Error analysis for fully discrete Galerkin method}
In analyzing this method, proceed in a similar fashion as for the semidiscrete method, we derive a error estimate for the fully discrete time Galerkin approximation the differential equation.
 
Let
$\rho^n(\cdot) = \rho(\cdot,t_n)$,  be the
true solution evaluated at the discrete time levels.  We will also
denote $\pi \rho^n \in W_h$ to be the projections of the true solutions at the discrete time levels.  As in the semidiscrete case, we use $\chi = \rho -\rho_h$, $\vartheta=\rho-\rho_h$, $\theta_h=\rho_h-\pi \rho$ and $\chi^n$, $\vartheta^n$, $\theta_h^n$  be evaluating $\chi$, $\vartheta$, $\theta_h$ at the discrete time levels.
We also define  
$$
\partial \phi^n = \frac {\phi^{n} -\phi^{n-1} }{\Delta t}.
$$

We rewrite  \eqref{weakform}  with $t=t_n$. Using the definition of $L^2$-projection and standard manipulations show that the true solution satisfies the discrete equation
\beq\label{fuldis1}
( \partial \pi\rho^n , w_h) +  \left(K(|\nabla \rho^n|)\nabla \rho^n, \nabla w_h\right)=(-\pi\rho_t^n+ \partial \pi\rho^n,w_h )-\langle \psi^n, w_h\rangle + (f^n, w_h), \forall w_h\in W_h. 
\eeq

\begin{theorem}
Let $\rho^n$ solve problem \eqref{weakform} and $\rho_h^n$ solve the fully  discrete Galerkin finite element approximation \eqref{fullydiscreteform} for each time step $n$, $n=1,\ldots, N$.  
There exists a positive constant $C$ independent of $h$ and $\Delta t$  such that if  the $\Delta t$ is sufficiently small then 
\beq\label{fulerrl2}
\max_{1\le n\le N}\norm{\rho^n-\rho_h^n} \le C\left( h^{\frac {r-1}2}+ \sqrt{\Delta t} \right).
\eeq  
\end{theorem}
\begin{proof}
Subtracting \eqref{fuldis1} from \eqref{fullydiscreteform} gives the error equation  
%
\beq\label{DiscreteErrEq}
(\partial \theta_h^n , w_h) +  \left(K(|\nabla \rho_h^n|)\nabla \rho_h^n- K(|\nabla \rho^n|)\nabla \rho^n, \nabla w_h\right)=(   \pi\rho_t^n-\partial \pi\rho^n,w_h ).  
\eeq  
Selecting $w_h =\theta_h^n,$ we rewrite above equation as form 
\beq\label{eereq}
\begin{split}
&(\partial \theta_h^n , \theta_h^n ) +  \left(K(|\nabla \rho^n|)\nabla \rho^n- K(|\nabla \rho_h^n|)\nabla \rho_h^n, \nabla \rho ^n- \nabla\rho_h^n \right)\\
&\hspace{2cm}=\left(K(|\nabla \rho^n|)\nabla \rho^n- K(|\nabla \rho_h^n|)\nabla \rho_h^n, \nabla \vartheta^n \right) +( \pi\rho_t^n -  \partial \pi\rho^n,\theta_h^n  ).  
\end{split} 
\eeq
For the first term, we have the identity 
\beq\label{rhs1}
\begin{split}
(\partial \theta_h^n , \theta_h^n )&=\left(\partial \theta_h^n, \frac{\theta_h^n +\theta_h^{n-1}}{2}+ \frac {\Delta t} 2 \partial \theta_h^n\right)\\
&= \frac 1{2\Delta t}\left(\norm{\theta_h^n}^2  -\norm{\theta_h^{n-1}}^2\right)+\frac{\Delta t}{2}\norm{\partial \theta_h^n }^2 . 
\end{split}
\eeq
For the second term, the monotonicity of $K(\cdot)$ in \eqref{Mono} yields 
 \beq\label{rhs2} 
 (K(|\nabla \rho^n|)\nabla \rho^n -K(|\nabla \rho_h^n|)\nabla \rho_h^n, \nabla\rho^n -\nabla\rho_h^n)  \ge C_0\omega^n \norm{\nabla \chi^n}_{L^{\beta}}^2 
 \eeq
 with 
 $
 \omega^n= \omega(t_n) = \left(1+\max\left\{\norm{\nabla \rho_h^n}_{L^{\beta}(\Omega)}, \norm{\nabla \rho^n}_{L^{\beta}(\Omega)} \right \}\right)^{-a} .
  $\\
For third term, using \eqref{i:ineq2}, Holder's inequality, there is a positive constant $C$  independence of $h, \Delta t, n$ such that    
\beq\label{lhsterm1}
\begin{split}
\left(K(|\nabla \rho^n|)\nabla \rho^n- K(|\nabla \rho_h^n|)\nabla \rho_h^n, \nabla \vartheta^n \right)\le C
\Mm ^{\frac 1{\beta\lambda}} \norm{\nabla \vartheta^n}_{L^\beta}.
\end{split}
\eeq
 For the last term, it follows from using $L^2$-projection and Taylor expand that  
  \beq\label{lhs1}
  \begin{split}
 ( \pi\rho_t^n -  \partial \pi\rho^n,\theta_h^n  )&=( \rho_t^n -  \partial\rho^n,\theta_h^n  ) =\left(\frac 1{\Delta t} \int_{t_{n-1}}^{t_n} \rho_{tt}(\tau) (\tau-t_{n-1})   d\tau, \theta_h^n \right).
 \end{split}
 \eeq
Differentiating equation \eqref{maineq} in time provides,   
  \beqs
\rho_{tt} = \nabla\cdot (K(|\nabla \rho|)\nabla \rho )_t+f_t.
\eeqs
This and \eqref{lhs1} imply  
  \beq\label{lhs11}
 \begin{split}
( \pi\rho_t^n -  \partial \pi\rho^n,\theta_h^n  )&= \left(\frac 1{\Delta t} \int_{t_{n-1}}^{t_n} f_t (\tau) (\tau-t_{n-1})   d\tau  , \theta_h^n \right)\\
&-\left(\frac 1{\Delta t} \int_{t_{n-1}}^{t_n}  (K(|\nabla \rho|) \nabla \rho )_t   (\tau-t_{n-1})   d\tau  , \nabla\theta_h^n \right)=I_1+I_2. 
 \end{split}
 \eeq

  For the first term of left hand side in \eqref{lhs11}, 
 \beq\label{lhs12}
 \begin{split}
 I_1&\le\frac 1{\Delta t}\norm{  \int_{t_{n-1}}^{t_n} f_t (\tau) (\tau-t_{n-1})   d\tau}   \norm{\theta_h^n} \\
 & \le\frac 1{\Delta t} \norm{f_t}_{L^2(I_n,L^2)} \left[\int_{t_{n-1}}^{t_n} (\tau-t_{n-1})^2   d\tau\right]^{\frac 12}   \norm{\theta_h^n}\\
 &\le \frac{1} 6 \Delta t\norm{f_t}_{L^2(I_n,L^2)}^2 + \frac12 \norm{\theta_h^n}^2,
 \end{split}
 \eeq
where $I_n=[t_{n-1},t_n].$  
 
For the second term of left hand side in \eqref{lhs11}, using integration by part, 
triangle inequality, \eqref{i:ineq2}, and H\"older's inequality we find that  
	 \beqs
	 \begin{aligned}
I_2&\le \frac 1{\Delta t}\left|\Big(\int_{t_{n-1}}^{t_n} K(|\nabla \rho|) \nabla \rho d\tau, \nabla\theta_h^n\Big)\right| +\left| \left(K(|\nabla \rho^n|) \nabla \rho^n, \nabla\theta_h^n  \right )\right|\\
& \le \frac C{\Delta t}\int_{t_{n-1}}^{t_n}  \left(|\nabla \rho|^{\beta-1} , |\nabla\theta_h^n|\right) d\tau +C(|\nabla \rho^n|^{\beta-1}, |\nabla\theta_h^n|)\\
& \le \frac C{\Delta t}\left[\int_{t_{n-1}}^{t_n}\norm{\nabla \rho }_{L^\beta}^{\beta-1}d\tau\right]\norm{\nabla\theta_h^n}_{L^{\beta}} +C\norm{\nabla\rho^n}_{L^{\beta}}^{\beta-1} \norm{\nabla\theta_h^n}_{L^{\beta}}.        
      \end{aligned}
      \eeqs
According to \eqref{mid4gradr},  
   \beq\label{lhs13}
   \begin{split}
      I_2 &\le C\Mm^{\frac {\beta-1}\beta}\norm{\nabla\theta_h^n}_{L^{\beta}}\\
      &\le C\Mm^{\frac {\beta-1}\beta}\norm{\nabla\chi^n}_{L^{\beta}}+C\Mm^{\frac{\beta-1}\beta}\norm{\nabla \vartheta^n}_{L^\beta}\\
      &\le \varep\norm{\nabla\chi^n}_{L^{\beta}}^2 + C\varep^{-1}\Mm^{2\frac {\beta-1}\beta}+C\Mm^{\frac{\beta-1}\beta}\norm{\nabla \vartheta^n}_{L^\beta}.
      \end{split}
      \eeq 
 Combining \eqref{lhs11}, \eqref{lhs12} and \eqref{lhs13} and using the fact that $\frac 1{\lambda} =\frac{\beta-1}{\beta}$,    
\beq\label{lhsterm2}
\begin{split}
 ( \pi\rho_t^n -  \partial \pi\rho^n,\theta_h^n  ) &\le \frac{1} 6 \Delta t\norm{f_t}_{L^2(I_n,L^2)}^2 + \frac12 \norm{\theta_h^n}^2\\
&\quad +\varep\norm{\nabla\chi^n}_{L^{\beta}}^2 + C\varep^{-1}\Mm^{\frac 2 \lambda}+C\Mm^{\frac 1 \lambda}\norm{\nabla \vartheta^n}_{L^\beta}.
 \end{split}
\eeq
It follows from \eqref{rhs1}, \eqref{rhs2}, \eqref{lhsterm1} and \eqref{lhsterm2}  that 
\beq\label{comb1}
\begin{aligned}
 \frac 1{2\Delta t}\left(\norm{\theta_h^n}^2  -\norm{\theta_h^{n-1}}^2\right)+  C_0\omega^n \norm{\nabla\chi^n}_{L^{\beta}}^2\le \varep \norm{\nabla \chi^n}_{L^{\beta}}^2 +C(\Mm ^{\frac 1{\beta\lambda}}+\Mm ^{\frac 1{\lambda}} ) \norm{\nabla \vartheta^n}_{L^\beta}\\
+ C\Delta t\norm{f_t}_{L^2(I_n,L^2)}^2+ \frac12 \norm{\theta_h^n}^2+ C\varep^{-1}\Mm^{\frac 2 \lambda}.
\end{aligned}
\eeq
Selecting $\varep =\frac{C_0\omega^n}2$ in \eqref{comb1}, we obtain  
\beq\label{comb1a}
\begin{aligned}
 &\frac 1{\Delta t}\left(\norm{\theta_h^n}^2  -\norm{\theta_h^{n-1}}^2\right)+  C_0\omega^n \norm{\nabla\chi^n}_{L^{\beta}}^2\le C(\Mm ^{\frac 1{\beta\lambda}}+\Mm ^{\frac 1{\lambda}} ) \norm{\nabla \vartheta^n}_{L^\beta}\\
 &\qquad+ C\Delta t\norm{f_t}_{L^2(I_n,L^2)}^2+  \norm{\theta_h^n}^2+ C(\omega^n)^{-1}\Mm^{\frac 2 \lambda}+C\Mm^{\frac 1 \lambda}\norm{\nabla \vartheta^n}_{L^\beta}.
\end{aligned}
\eeq
Dropping the second term in \eqref{comb1a} and using \eqref{Odef}  and the fact $\frac 2 \lambda+\gamma=1$, we get
\beqs
\begin{split}
 \frac 1{\Delta t}\left(\norm{\theta_h^n}^2  -\norm{\theta_h^{n-1}}^2\right)&\le \norm{\theta_h^n}^2+ C \left\{ (\Mm ^{\frac 1{\beta\lambda}} +\Mm^{\frac 1\lambda })\norm{\nabla \vartheta^n}_{L^\beta}+ \Mm^{\frac 2 \lambda+\gamma}\right\} + C \Delta t\norm{f_t}_{L^2(I_n,L^2)}^2.
\end{split}
\eeqs
Summing from $n=0$ to $n=m-1$ with $0<m\le N$ and using the fact that $\theta_h(0)=0$, we obtain  
\beq\label{comb2}
\begin{aligned}
 \norm{\theta_h^m}^2 \le \Delta t\sum_{n=0}^{m-1} \norm{\theta_h^n}^2 &+ C\Big\{ (\Mm^{\frac{1}{\beta\lambda}} +M^{\frac 1\lambda}) \sum_{n=0}^{m-1} \Delta t \norm{\nabla \vartheta^n}_{L^{\beta} }\\
 & + (\Delta t)^2\sum_{n=0}^{m-1}\norm{f_t}_{L^2(I_n,L^2)}^2 +(m-1)\Mm \Delta t\Big\}.
\end{aligned}
\eeq
An application of the discrete Gronwall's inequality shows that for $\Delta t$ sufficiently small,   
\beq\label{comb}
\begin{aligned}
\max_{0\le n\le N}\norm{\theta_h^n}^2  &\le C\Big\{(\Mm ^{\frac 1{\beta\lambda}} +\Mm^{\frac 1\lambda }) \sum_{n=0}^{m-1} \Delta t \norm{\nabla \vartheta^n}_{L^{\beta} } +\Mm \Delta t\Big\}+C(\Delta t)^2\norm{f_t}_{L^2(I,L^2)}^2,
\end{aligned}
\eeq
 which implies 
 \beq\label{comb4}
\begin{aligned}
\max_{0\le n\le N}\norm{\theta_h^n} &\le C\Big\{ (\Mm ^{\frac 1{2\beta\lambda}} +\Mm^{\frac 1{2\lambda} }) h^{\frac{r-1} 2}\Big(\sum_{n=0}^{N-1}\Delta t\norm{\rho^n}_{W^{r,\beta}}\Big)^{\frac 1 2}+ \Mm^{\frac 1 2}\sqrt{\Delta t}  \Big\}\\
&\quad+ C\Delta t\norm{f_t}_{L^2(I,L^2)}.
\end{aligned}
\eeq 
Finally, the result follows from \eqref{comb4} and   
$
 \norm{\rho^n-\rho_h^n}\le \norm{\theta_h^n} +\norm{\vartheta^n}. 
$
 \end{proof}
 
Now we give the error estimate for gradient in fully discrete version.      

\begin{theorem} Let $\rho$ solve problem \eqref{weakform} and $\rho_h^n$ solve the fully  discrete Galerkin  finite element approximation \eqref{fullydiscreteform} for each time step $n$, $n=1,\ldots, N$.  Then there exists a positive constant $C$ independent of $h$ and $\Delta t$ satisfying
\beq\label{ErrGradFul}
\norm{\nabla \rho^n_h - \nabla \rho^n}_{L^\beta}\le C ( h^{\frac {r-1} 4} +\sqrt[4]{\Delta t}). 
\eeq
 
\end{theorem}

\begin{proof}
Selecting $w_h =\theta_h^n,$ in \eqref{DiscreteErrEq} and rewriting the resulting equation as form   
\begin{multline*}
 \left((K(|\nabla \rho^n|)\nabla \rho^n- K(|\nabla \rho_h^n|)\nabla \rho_h^n, \nabla \rho ^n- \nabla\rho_h^n \right)\\
=\left(K(|\nabla \rho^n|)\nabla \rho^n- K(|\nabla \rho_h^n|)\nabla \rho_h^n, \nabla \vartheta^n \right) +( \rho_t^n -  \partial \rho_h^n,\theta_h^n  ).  
\end{multline*} 
Due to \eqref{Kcont}, Cauchy-Schwartz and triangle inequality, one has  
\beqs
\omega^n\norm{\nabla \rho^n_h - \nabla \rho^n}_{L^\beta}^2\le C \left(\norm{ \partial\rho_h^n} +\norm{\rho_t^n} \right)\norm{ \theta_h^n} +C(|\nabla \rho^n -\nabla \rho_h^n |, |\nabla\vartheta^n| ). 
\eeqs
Since
\begin{align*}
\norm{\partial \rho_{h}^n} =(\Delta t)^{-1}\norm{ \int_{t_{n-1}}^{t_n} p_{h,t} dt}
 \le (\Delta t)^{-1} \int_{t_{n-1}}^{t_n} \norm{p_{h,t}} dt\le \max_{[T/N,T]}\norm{p_{h,t}}\le \Dd^2, 
\end{align*}
$ 
\norm{p_t^n} \le \max_{[T/N,T]}\norm{p_t}\le \Dd^2 
$ and \eqref{lhsterm1}, we obtain      
\beqs
\omega^n\norm{\nabla \rho^n_h - \nabla \rho^n}_{L^\beta}^2\le C \Dd^2 \norm{ \theta_h^n} + C\Mm^{\frac 1{\beta\lambda}} \norm{\nabla\vartheta^n}_{L^\beta}. 
\eeqs
Using estimate in \eqref{comb4} and \eqref{Odef} shows that
\beqs
\begin{split}
\norm{\nabla \rho^n_h - \nabla \rho^n}_{L^\beta}^2&\le C \Dd^2 \Mm^{\gamma} \Big\{ (\Mm^{\frac 1{2\beta\lambda}}+\Mm^{\frac 1{2\lambda} } ) h^{\frac{r-1}2}\Big(\sum_{n=0}^{N-1} \Delta t \norm{\rho^n}_{W^{r,\beta}}\Big)^{\frac 1 2}+ \Mm^{\frac 1 2}\sqrt{\Delta t} \\
&\quad + \Delta t\norm{f_t}_{L^2(I,L^2)} \Big\}+ C\Mm^{\gamma+\frac 1{\beta\lambda}} h^{r-1}\norm{\rho^n}_{W^{r,\beta}}, 
\end{split}
\eeqs
which proves \eqref{ErrGradFul}. The proof is complete. 
\end{proof}
\section{Numerical results} \label{Num-result}

In this section, we give a two simple numerical experiments using Galerkin finite element method in the two dimensional region to illustrate the convergent theory. For simplicity, we test the convergence of our method with the Forchheimer two-term law $g(s)=1+s$. Equation \eqref{Kdef} $sg(s)=\xi,$ $s\ge0$ gives
 $ s= \frac {-1 +\sqrt{1+4\xi}}{2}$ and hence 
$$
K(\xi) =\frac1{g(s(\xi)) } =\frac {2}{1+\sqrt{1+4\xi}}.
$$ 
The region is selected is unit square, i.e $\Omega=[0,1]^2$ . We use the Lagrange element of order $r=2$ on the unit square in two dimensions. We used FEniCS \cite{fenics} to perform our numerical simulations. We divide the unit square into an $ N\times N$ mesh of squares, each then subdivide into two right triangles using the \textsf{UnitSquareMesh} class in FEniCS. For each mesh, we solve the generalized Forchheimer equation numerically. The error control in each nonlinear solve is $\varep =10^{-6}$.  Our problem is solved at each time level start at $t=0$ until final time $T = 1$. At this time, we measured the $L^2$-errors of pressure and $L^{\beta}$-errors of gradient of pressure and velocity. Here $\beta = 2 - a = 2- \frac{{\mathrm deg} (g)}{{\mathrm deg} ( g) + 1}=\frac 32$.

{\bf Example 1.} The analytical solution is as follows  
\beqs
\rho(x,t)=e^{-2t}\left[\frac 12 (x_1^2+x_2^2)-\frac 13(x_1^3-x_2^3) \right]+1, \quad \forall (x,t)\in \Omega\times[0,1].  
\eeqs
The forcing term $f$ is determined accordingly to the analytical solution by equation $p_t - \nabla \cdot (K(|\nabla \rho|)\rho ) = f$. Explicitly,  
\beqs
\begin{split}
 f(x,t)&=-2e^{-2t}  \Big[\frac12 (x^2_1 +x^2_2)-\frac 13(x^3_1+x^3_2)\Big]-\frac {4e^{-2t}(1-x_1-x_2)}{1+\sqrt{1+4e^{-2t}W(x)  }}\\
&+\frac{4e^{-4t}\Big[ x_1^2(1-x_1)^2(1-2x_1)+ x_2^2(1-x_2)^2(1-2x_2) \Big] }{ W(x) \left(1+\sqrt{1+4e^{-2t}W(x)}\right)^2 \sqrt{1+4e^{-2t}W(x)} }.
\end{split}
\eeqs   
where $W(x) = \sqrt{x_1^2(1-x_1)^2+x_2^2(1-x_2)^2}$. The initial data $\rho^0(x) =\frac 12 (x_1^2+x_2^2)-\frac 13(x_1^3-x_2^3)+1$ and the Neumann boundary condition $\psi(x,t)=0.$
 The numerical results are listed as the following table.   
 \vspace{0.3cm}  
\begin{center}
\begin{tabular}{l| c| c| c| c}
\hline
N    & \qquad $\norm{\rho-\rho_h}$ \qquad  &  \quad Rates  \quad&  \quad$\norm{\nabla(\rho-\rho_h)}_{L^{\beta}(\Omega)}$ \quad & \quad  Rates \quad  \\
\hline
4	& 6.33e-02   &-          	&  4.51E-01		&-\\
8	&5.50E-02		&0.20  		& 4.07E-01  			&0.15\\
16	&4.52E-02		&0.28    	& 3.70E-01 		&0.14\\
32	&3.50E-02		&0.37  		& 3.22E-01 		&0.20\\
64	&2.53E-02		&0.47		&  2.70E-01		&0.25\\
128&1.73E-02		&0.55		& 2.21E-01 		&0.29\\
256&1.13E-02		&0.61		&  1.79E-01		&0.31\\
512&7.19E-03		& 0.63		&  1.43E-01		&0.32\\
\hline
\end{tabular}
\vspace{0.3cm}

Table 1. {\it Convergence study for generalized Forchheimer flows using Galerkin finite element method with zero flux on the boundary in 2D.}

\end{center}

{\bf Example 2.} The analytical solution is $\rho(x,t)=x_1x_2e^{-t}+1$ for all  $(x,t)\in \Omega\times [0,1]$. The forcing term $f$, initial condition and Neumann boundary condition are determined accordingly to the analytical solution as follows 
\beqs
 f(x,t)= -e^{-t}x_1x_2+\frac{8e^{-2t}x_1x_2}{\sqrt{x_1^2+x_2^2}\left(1+\sqrt{1+4e^{-t}\sqrt{x_1^2+x_2^2}}\right)^2\sqrt{1+4e^{-t}\sqrt{x_1^2+x_2^2}}},
\eeqs   
and
\beqs
\rho^0(x) =x_1x_2+1, \quad \psi(x,t)=\frac{2e^{-t}} {1+\sqrt{1+4e^{-t}\sqrt{x_1^2+x_2^2}} }
\begin{cases} 
x_2 &\text { on } x_1=0,\\
-x_2 & \text { on } x_1=1,\\
-x_1 & \text { on } x_2=1,\\
x_1 & \text { on } x_2=0.
 \end{cases}
\eeqs
The numerical results are listed in below table 
 \vspace{0.3cm}  
\begin{center}
\begin{tabular}{l|c| c|c| c}
\hline
N    &  \qquad $\norm{\rho-\rho_h}$   \quad &   \quad Rates  \quad & $ \quad \norm{\nabla(\rho-\rho_h)}_{L^{\beta}(\Omega)}$  \quad  &  \quad  Rates  \quad   \\
\hline
4	&4.40E-02     	&-           	& 2.67E-02 		&-\\
8	&2.24E-02		&0.97  		& 2.02E-02 		&0.40\\
16	&1.15E-02		&0.96    	& 1.37E-02 		&0.56\\
32	&5.90E-03		&0.96  		& 8.53E-03 		&0.68\\
64	&3.01E-03		&0.97		& 4.99E-03 		&0.77\\
128&1.53E-03		&0.98		& 2.79E-03 		&0.84\\
256&7.70E-04		&0.99		& 1.52E-03 		&0.88\\
512&3.94E-4		&0.99		&8.28E-4  		 	&0.92\\
\hline
\end{tabular}
\vspace{0.3cm}

Table 2. {\it Convergence study for generalized Forchheimer flow using Galerkin finite element method with nonzero flux on the boundary in 2D.}
\end{center}

\medskip
\textbf{Acknowledgment.} The author wishes to thank Dr. Luan Hoang for valuable assistance in discussion and suggestions.


\bibliographystyle{siam}

\end{document}